\documentclass[a4paper,11pt]{article}
\title{Robustness of ergodic properties of nonautonomous piecewise expanding maps}
\author{Matteo Tanzi\footnote{m.tanzi13@imperial.ac.uk}, ~ Tiago Pereira\footnote{tiago.pereira@imperial.ac.uk}, and ~  Sebastian van Strien\footnote{s.van-strien@imperial.ac.uk}  \\
Department of Mathematics, Imperial College London\\ }
\date{}
\usepackage{fullpage}
\usepackage[british]{babel}
\usepackage{wrapfig}
\usepackage[square,sort,comma,numbers]{natbib}
\usepackage[titletoc,title]{appendix}
\usepackage[utf8]{inputenc}
\usepackage{amssymb}
\usepackage{amsthm}
\usepackage{graphics}
\usepackage{dsfont}
\usepackage{amsmath}
\usepackage{amstext}
\usepackage{engrec}
\usepackage{floatflt}
\usepackage{rotating}
\usepackage[safe,extra]{tipa}
\usepackage[font={footnotesize}]{caption}
\usepackage{subcaption}
\usepackage{framed}
\usepackage{listings}
\newcommand{\diff}{\mathrm{d}}
\newcommand{\mc}{\mathcal}
\newcommand{\mb}{\mathbb}

\newcommand{\R}{\mb R}
\newcommand{\C}{\mb C}
\newcommand{\N}{\mb N}

\newcommand{\eea}{\end{align}}

\renewcommand{\epsilon}{\varepsilon}
\renewcommand{\bar}{\overline}
\renewcommand{\tilde}{\widetilde}

\newcommand{\bo}{\boldsymbol}
\renewcommand{\phi}{\varphi}

\def\hg{{\hat \gamma}}
\newtheorem{theorem}{Theorem}
\newtheorem{theo}{Theorem}

\newtheorem{corollary}{Corollary}
\newtheorem{lemma}{Lemma}
\newtheorem{proposition}{Proposition}
\theoremstyle{definition}
\newtheorem{definition}{Definition}
\theoremstyle{remark}
\newtheorem{remark}{Remark}
\newtheorem{example}{Example}

\newtheoremstyle{algorithm}
{4pt}
{4pt}
{}
{}
{}
{:}
{\newline}
{}

\newtheorem{algorithm}{Algorithm}

\newcommand{\balgorithm}{\begin{algorithm}\begin{framed}\ }
\newcommand{\ealgorithm}{\end{framed}\end{algorithm}}

\newcommand{\pippo}{\begin{code} \  \begin{lstlisting}[frame=single]}
\newcommand{\pappo}{\end{lstlisting} \end{code}}

\newcommand{\bd}{\begin{definition}}
\newcommand{\ed}{\end{definition}}
\newcommand{\bt}{\begin{theorem}}
\newcommand{\et}{\end{theorem}}
\newcommand{\bp}{\begin{proposition}}
\newcommand{\ep}{\end{proposition}}
\newcommand{\bc}{\begin{corollary}}
\newcommand{\ec}{\end{corollary}} 
\newcommand{\bl}{\begin{lemma}}
\newcommand{\el}{\end{lemma}}
\newcommand{\br}{\begin{remark}}
\newcommand{\er}{\end{remark}}

\DeclareMathOperator{\spec}{spec}
\DeclareMathOperator{\osc}{osc}
\DeclareMathOperator{\essup}{Esup}
\DeclareMathOperator{\esinf}{Einf}
\DeclareMathOperator{\supp}{supp}
\DeclareMathOperator{\clos}{Clos}
\DeclareMathOperator{\inter}{Int}
\DeclareMathOperator{\diam}{diam}

\begin{document}
\maketitle

\begin{abstract}
Recently, there has been an increasing interest in nonautonomous composition of perturbed hyperbolic systems: composing
 perturbations of a given hyperbolic map  $F$ results in statistical behaviour close
to that of $F$. We show this fact in the case of piecewise regular expanding maps. In particular, we impose conditions on perturbations of this class of maps that include situations slightly more general than what has been considered so far, and prove that these are stochastically stable in the usual sense. We then prove that the evolution of a given distribution of mass under composition of time dependent perturbations (arbitrarily - rather than randomly - chosen at each step) close to a given map  $F$ remains close to the invariant mass distribution of $F$. Moreover, for almost every point, Birkhoff averages along trajectories do not fluctuate wildly. This result complements recent results on memory loss for nonautonomous dynamical systems.
\end{abstract}

\section{Introduction}

During the last decade there has been an increasing focus on nonautonomous dynamical systems meaning that, rather than iterating a single map $F:M\rightarrow M$ on the given phase space $M$, one looks at the evolution under composition of several different self-maps of $M$. The main motivation for this point of view is that in practical applications, the map $F_t$ describing 
how the state variable evolves from time $t$ to time $t+1$ should depend on $t$. In this paper, the only assumption we make on $F_t$  
 is that at each time it is close to a given map $F$.
So we will  {\em not} assume that the choice of the maps $F_t$ follows some random distribution
nor that the system can be written as a skew-product. 

Nonautonomous composition of small perturbations of  a given dynamical system can lead to the pointwise destruction of typical statistical behaviour. For example, in \cite{MR1041523} the authors show that for any point in phase space one can construct a sequence of time dependent perturbations that makes the point evolve arbitrarily close to a periodic orbit of the unperturbed map. This means that even if the Birkhoff averages along the unperturbed orbit were typical, this is not the case for the nonautonomous dynamics. We would like to answer the question whether under the same nonautonomous evolution, the change of statistical behaviour by small perturbations  can be observed on a set of positive measure  (w.r.t. the reference measure). We address the problem in the case of multidimensional piecewise maps \cite{Saus}.   

We first define a collection of perturbations to a given multidimensional piecewise expanding map $F$, and we prove that the system is stochastically stable
for this type of perturbations (meaning that perturbed maps with perturbations of given small magnitude have an invariant density which is close to the invariant density for the unperturbed system). Then we prove that the evolution of sufficiently regular mass distributions under the time-dependent dynamics become, up to a fixed precision, close to the mass distribution  which is invariant under $F$. We then use this result together with a law of large numbers for dependent random variables to prove that, given sufficiently regular observables, for almost every point the accumulation points of the sequence of Birkhoff averages is close to the expectation of the observable with respect to the invariant measure.

Our results can be applied to certain dynamical systems defined on networks  whose  topology slightly changes over time. In  applications, we have in mind some of the edges of the network are occasionally broken due to mechanical failures. Under certain settings we show that such intermittent mechanical failures do not significantly change the ergodic properties of the system.

The paper is organised as follows. In Section~\ref{sec:statements} we state the main results and discuss the existing literature. In Section~\ref{Sec:examples} we  give some applications of the main result, in particular to dynamics on networks with changing topologies.   In Section~\ref{sec:prfthmB} we give the precise definitions and prove the main theorem. The proof we use relies on  spectral stability. In the appendix, we formulate a related result for $C^{1+\nu}$-expanding maps, presenting a different approach relying on invariant cones. We provide an extensive outline the ideas of the proofs of both results.

\section{Statements of the Results}\label{sec:statements}

We consider the class of maps introduced by \cite{Saus}.   The phase space is $\Omega\subset\R^N$, a compact subset of $\R^N$, which is decomposed in a fixed number of domains (allowed to slightly change in the perturbed versions of the maps). The domains can have fractal boundaries, 
and the restriction to each of them is regular. The precise hypotheses that a map $F:\Omega\rightarrow\Omega$ and its perturbations 
 must satisfy are given in properties (ME1)-(ME6) (Section \ref{sec:prfthmB}) and  (CM1)-(CM2) (Section \ref{PertPiecCaseSec}). 
Under these assumptions we obtain the following
 
\begin{theo}
\label{thmA}

Let $F_{\hg}$ be a piecewise expanding map on the compact set $\Omega\subset\R^N$ belonging to a collection of maps $\{F_{\gamma}\}_{\gamma \in \Gamma}$ satisfying (ME1)-(ME6), continuous at $\hg\in \Gamma$ ((CM1)-(CM2)). Then, 
for each  $\epsilon>0$ there exists $\delta>0$ so that:
\begin{enumerate}
\item[(1)] if $\nu$ is a Borel probability measure on $\Gamma$ with $\supp \nu\subset B_\delta(\hg)$, then there exists $\phi_\nu$ stationary density for the random dynamical system obtained composing independently maps of the family according to $\nu$ and it satisfies
$$
\|\phi_\nu-\phi_{\hg}\|_1\leq\epsilon,
$$
in particular, $\forall \gamma\in B_{\delta}(\hg)$,  $F_\gamma$ has an invariant density $\phi_\gamma$ and
$$
\|\phi_\gamma-\phi_{\hg}\|_1\leq \epsilon;
$$
\item[(2)]  if $\bo \gamma\in B_\delta(\hg)^{\N}$ then for every probability measure $\mu=\phi m$ with density $\phi\in V_\alpha$ ($V_\alpha\subset L^1$ is defined in 
(\ref{QuasHoldSpac}) ), 
 there exists $\bar n:=\bar  n(\epsilon,\phi)\in\N$ such that for every $n>\bar n$  the Radon-Nykodim derivative $\frac{d}{dm}({F^n_{\bo \gamma}}_*\mu)$ has a representative in $V_\alpha$, and
\begin{equation}
\left\|\frac{d}{dm}(F^n_{\bo \gamma})_*\mu-\phi_{\hg}\right\|_1<\epsilon;
\label{asympt-dens} \end{equation}
\item[(3)]
moreover, for any observable  $\psi \in V_{\alpha}$ there exists a set $X_{\bo \gamma}$ of full measure so that for every $x\in X_{\bo \gamma}$ 
$$
\int \psi   d \mu_{\hg}-\epsilon \mb \|\psi \|_1 \leq \liminf_{n\rightarrow\infty}\frac{1}{n}S_n(\psi)(x)\leq \limsup_{n\rightarrow\infty}\frac{1}{n}S_n(\psi)(x)\leq \int \psi d\mu_{\hg}+\epsilon\mb \| \psi \|_1
$$
where $\mu_{\hg} = \phi_{\hg}m$ and $S_n(\psi)(x)=\psi(x)+\sum_{i=1}^{n-1}\psi\circ F_{\gamma_i}\circ... \circ F_{\gamma_1}(x)$.
\end{enumerate}
\end{theo}

Part (1) of the theorem proves stochastic stability of the collection and in particular proves continuous
dependence of invariant measures w.r.t. the size of the perturbation. Part (2) shows that the evolution $(F^n_{\bo \gamma})_*\mu$ of the mass $\mu$ with density $\phi$ remains close
to the invariant measure (indeed, this holds in the $L^1$ sense for the densities). For many applications one cannot be confident that $\bo \gamma$
is chosen randomly, and therefore the assertion from Part (2) is more useful 
than having merely stochastic stability.  Part (3) shows that one has quasi-Birkhoff behaviour for 
time averages, meaning that the accumulation points of the time averages remain close to the
space average of the observable w.r.t. the invariant measure of the unperturbed system.

\begin{remark}
Previous results on stochastic stability \cite{Cow} for piecewise expanding maps require
the domains of the partition elements to have piecewise smooth boundaries. 
\end{remark}
\begin{remark}
In the appendix we also formulate a related result in the context of $C^{1+\nu}$-expanding maps, using the
contraction properties of the transfer operator on suitable cones. The improved regularity 
results in uniform estimate, rather than the $L^1$ estimate in Part (2) above.
We will discuss previous results, give an overview of the literature  and also the ideas of the proofs at the end of this section.
\end{remark}

\begin{remark}
Keller \cite{KellPro} obtained a result $L^1$-analogue of  inequality (\ref{asympt-dens})  for piecewise expanding interval maps. 
\end{remark}

\begin{corollary}\label{Cor:WeakTop}
Let $\mu_\hg:=\phi_\hg m$ be the invariant measure for $F_\hg$, then for every neighbourhood $\mc U(\mu_\hg)$ w.r.t. the weak topology there is $\delta>0$ such that for every sequence $\bo\gamma\in B_\delta(\hg)^{\N}$, and for almost every $x\in \Omega_{\bo\gamma}$, there is $\bar n$ such that  
$$
\frac{1}{n}\sum_{i=0}^{n-1}\left(F_{\bo \gamma}^n\right)_*\delta_{x}\in \mc U(\mu_\hg),\quad \forall n>\bar n.
$$ 
\end{corollary}

\subsection{Historical comments}\label{subsec:review}
Certain classes of dynamical systems possess good statistical properties under the effect of small random perturbations. For example, stochastic perturbations 
(i.e. chosen randomly according to some distribution) of expanding maps on a compact manifold have been thoroughly investigated (e.g. \cite{ViaSdds,AlvArau,ArauTah,BalYou}), as well as those of piecewise expanding maps on the interval (e.g. \cite{LasYor,Kel,Liv5}). Some recent studies deal with maps of the interval with neutral fixed points \cite{SebShe} and  with multi-dimensional piecewise maps of compact subsets in $\mathbb{R}^N$ \cite{BoyGor2,Cow, Saus}. In all these cases one can give a description of the statistical behaviour of the orbits of the randomly perturbed system via a stationary measure which is also close to some absolutely continuous invariant density for the unperturbed system. Key to these results is that perturbations are independent and identically distributed. However, recent developments requires the understanding of the asymptotic behaviour of dynamical systems under nonautonomous perturbations (\cite{Martin}). These perturbations are not independent, and the natural question concerns whether the statistics of the unperturbed and perturbed maps remains close in this more general setting. We provide an affirmative answer for two classes of dynamical systems: piecewise $C^{1+\nu}$ maps of a compact subset of $\mathbb{R}^N$ and (in the appendix) $C^{1+ \nu}$ expanding maps of a compact manifold.

Recent work (among others \cite{ConRau}, \cite{Aim}, and \cite{Aim2}) on nonautonomous composition of dynamical systems (sometimes also referred to as sequential dynamical systems) focused mainly on proving that the system exhibits memory loss, which roughly means that, given a finite precision, the orbits of sufficiently regular densities of states become indistinguishable after a finite number of iterations of the system. They
 show that if the density of the measures  $\mu,\nu$ belong to a suitable cone then one has {\em memory loss}:
$$\left\| \frac{d}{dm}(F^n_{\bo \gamma})_*\mu- \frac{d}{dm} (F^n_{\bo \gamma})_*\nu\right\|_1 \to 0$$
as $n\to \infty$. Memory loss under nonautonomous perturbations holds also for example for contracting systems, where all orbits tend to get indefinitely close thus losing track of their initial condition. It is easy to give examples where one has memory loss, where 
the densities $\frac{d}{dm}(F^n_{\bo \gamma})_*\mu$ strongly fluctuate. Part (2) of  Theorem A shows that the size of the fluctuations, in the above setting, only depend on the size of the perturbation. 

 A work similar in spirit to ours is \cite{KellPro} where the author derives a result on perturbed operators satisfying the hypotheses of an ergodic theorem by Ionescu-Tulcea and Marinescu (Theorem \ref{ThmIonTulcMar}, \cite{IonTul}) and applies it to one-dimensional piecewise expanding maps.  The same approach could be used to deal with the multidimensional case imposing conditions on the maps and their perturbations so that they fit into the hypotheses of the theorem. We follow a slightly different argument. Also in \cite{nandori2012central} the authors provide general conditions on transfer operators and on observables that ensure the validity of a central limit theorem for Birkhoff sums. In \cite{haydn2017almost} analogous conditions are provided for the almost sure invariance principle to hold. Other notable works are \cite{Mikko, MR2842105,stenlund2013dispersing} where a coupling technique is used to prove exponential memory loss for the evolution of densities in the smooth expanding case, in the one-dimensional piecewise expanding case, in the two-dimensional Anosov case and in Sinai billiards with slowly moving scatterers. In \cite{dobbs2016quasistatic}, the authors look at the nonautonomous composition of one-dimensional expanding maps which are changing very slowly in time. In this situation, they can describe Birkhoff sums and their fluctuations as diffusion processes in the adiabatic limit of very slow change. It is also worth noticing that in \cite{MR1938476}, the authors define Banach spaces that allow to give a complete picture of the spectral properties of Anosov systems. Combining this result with the result from \cite{KellPro}, one can obtain robustness of the evolution of densities. 
 While we were writing this paper we also came across \cite{Gup} where memory loss is discussed in the multidimensional piecewise expanding setting for strongly mixing systems, but using techniques closer to \cite{Liv5}. 

\subsection{Strategy of the Proof}\label{Sec:strat} 

 To prove Theorem A one could proceed in a similar way looking at the application of a nonautonomous sequence of perturbed transfer operators on some invariant cone of functions with finite dimeter. This is the approach followed in \cite{Gup} to prove memory loss for the nonautonomous composition. However, to proceed in this way, one needs to restrict to a composition of maps which have a strong mixing property on the whole phase space. This is required because otherwise the support of the invariant density for the unperturbed map and for an arbitrarily small perturbation might not coincide, making the diameter of any cone of functions containing both infinite with respect to the Hilbert metric.  

Since we want to treat the more general case we will use another technique that will work for any nonautonomous composition of perturbed versions of a piecewise expanding map with quasi-compact transfer operator having 1 as unique simple eigenvalue. The restriction of the transfer operators to a Banach subspace of $L^1$ made of quasi-H\"older functions, as defined in \cite{Saus}, satisfy a Lasota-Yorke inequality, which implies quasi-compactness and the presence of a spectral gap. Each transfer operator thus induce a splitting on the space $ V_\alpha$ that can be written as the direct sum of a one-dimensional eigenspace corresponding to the invariant density, and the subspace of quasi-H\"older functions with zero mean, and the successive application of the transfer operator on a probability density, makes it converge (with respect to the $L^1$ norm) exponentially fast to the associated invariant density. One then shows that the invariant density is stochastically stable, implying that sufficiently small perturbations of a given piecewise expanding map have invariant densities close in the $L^1$ norm. The main idea to prove the result is to follow the evolution of a probability density splitting at each iteration on the eigenspaces corresponding to the invariant densities, and to control the remainders that this projections introduce using consequences of Lakota-Yorke inequality and spectral properties. We treat all this in Section \ref{sec:prfthmB}, where we first introduce a precise definition of the maps we considered, followed by a preliminary section (Section \ref{prelB}) on spectral and perturbation results. We introduce perturbations and perturbative results in Section \ref{PertPiecCaseSec} to conclude with the proof of the main claim in Section \ref{SecProfMaiTheMul}.

An alternative way to prove Part (2) of Theorem A would have been showing that the setting taken from \cite{Saus} with the perturbation we introduce can be framed in the general theorem of \cite{KellPro}. To do so one should prove that the operators acting on the Banach subspace of $L^1$, $V_\alpha$ (or some possibly larger space), satisfy all the hypotheses of the abstract theorem. This has been done by Keller in the one-dimensional case considering the properties of the restriction of transfer operators of one-dimensional piecewise expanding maps to the space of functions with bounded variation.

We deduce information on the decay of correlations of random variables $\psi_i=\psi\circ F_{\gamma_n}\circ...\circ F_{\gamma_1}$ from the spectral properties. We use these properties to apply a strong law of large numbers (\cite{Walk2004}) for correlated random variables. This gives an expression for the accumulation points of the Birkhoff averages and enables us to obtain part (3) of Theorem A.

\bigskip

The appendix of this paper contains Theorem B which treats the  $C^{1+\nu}$-expanding setting using cones and the Hilbert metric. The outline of the proof is given in the appendix.

\section{Applications }\label{Sec:examples}

\paragraph{Dynamics on non-stationary networks.}  A typical application of Theorem A we have in mind is that of dynamics on networks of fixed size $n$
in which the topology of the network or the coupling strength is allowed to fluctuate  over time. 
For example consider the following dynamics:
\begin{equation}
{x}_i(t+1) = {f}({ x}_i(t)) + \alpha \sum_{j=1}^n A_{ij}(t) h_{ij}(t, x_j(t), x_i(t))\, , \,\, \, \,\mbox{ for } i=1,\dots,n, t\in \mathbb N
\label{md1}
\end{equation}
where $f\colon T^n\to T^n$ is a  expanding map on the $n$-dimensional torus,
$x_i(t)$ describe the state of the $i$-th node at time $t$,   $\alpha\in \mathbb R$ describes the overall coupling strength,
$A_{ij}(t)\in \{0,1\}$ is the adjacency matrix of the network (so describes whether or not node $i$ is  connected to node $j$) 
and $h_{ij}\colon {\mathbb N} \times T^n\times T^n\to T^n$ is a time dependent coupling map.
For the uncoupled case $\alpha=0$, the system has an invariant measure which is absolutely continuous w.r.t. the Lebesgue measure. 
Our results show that, provided $|\alpha|$ is small, a regularly distributed collection of initial points remains
almost uniformly distributed as time progresses, even when the topology of the network changes at each time step, as in Figure 1.
Independence on the time $t$ is often an unreasonable assumption. If a connection between $i$ and $j$ 
is broken at time $t$, it most likely will take time to be fixed. For this reason our theorem
would apply in this setting, whereas standard results on stochastic stability would not.

\begin{figure}[t]
   \centering
   \includegraphics[width=7cm]{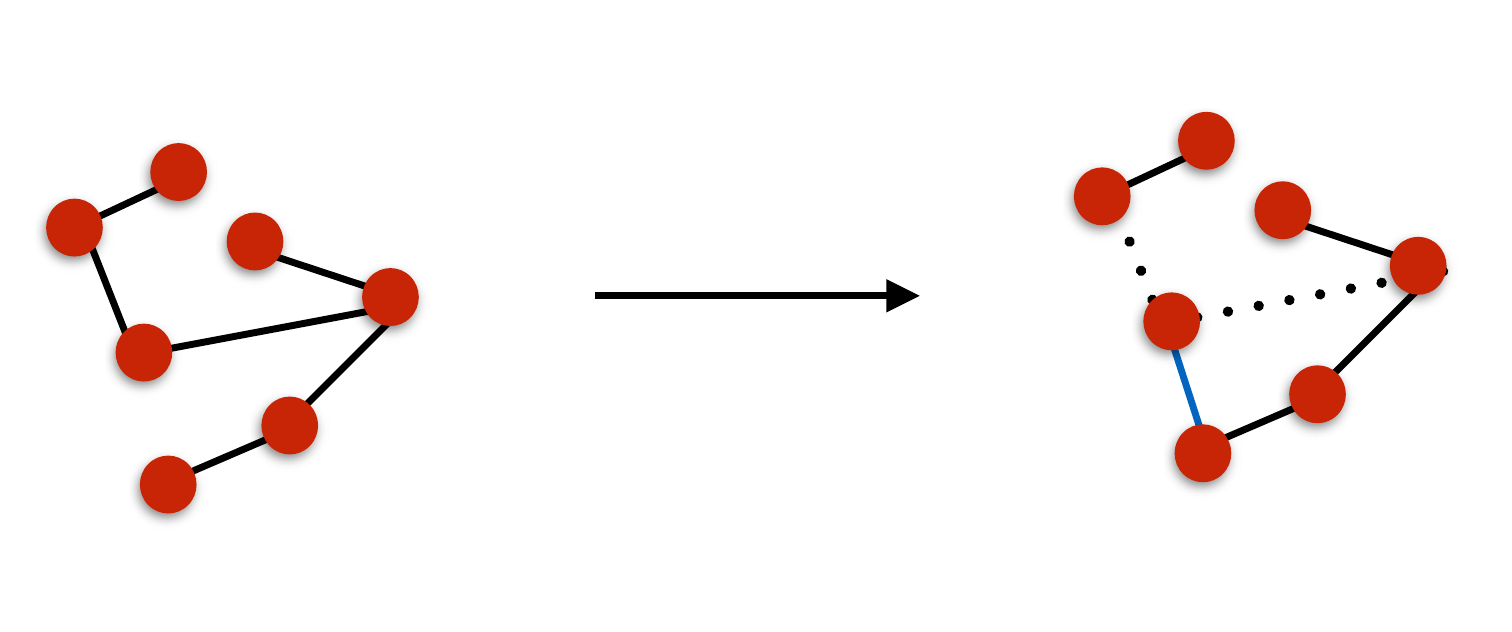} 
   \caption{Networks with changing topology over time. During the time step depicted above, two edges are deleted (dashed lines) while one is added. }
   \label{fig1}
\end{figure}

\paragraph{Stochastic Stability does not imply robustness of measures.}

In this example we show that a stochastically stable dynamical system, that is, a system which admits a stationary measure under 
random independent perturbations, can have 
 intricate behaviour under nonautonomous perturbations. For instance, take $\kappa\in (0,1)$ 
 and perturbations of the Pomeu-Manneville  map on the circle
\[
f_\gamma(x)=x+x^{1+\kappa}+\gamma x \mod 1
\]
or of the 
Liverani-Saussol-Vaienti circle map
\[
f_\gamma(x)= \left\{ \begin{array}{ll}  x(1+2^\kappa x^\kappa) +\gamma x& \mbox{ for }x\in [0,1/2), \\
2x-1 +\gamma x & \mbox{ for }x\in [1/2,1), \end{array} \right.
\]
see Fig \ref{Fig2} for an illustration. For $\gamma=0$ these maps have an absolutely continuous invariant measure, but for 
 $\gamma<0$ close to zero, $f_\gamma$ have a stable attracting fixed point. In spite of this, 
 Shen and van Strien proved that such maps are stochastically stable. 
 \begin{figure}[b]
  \begin{subfigure}{0.5\textwidth}
  \centering
  \includegraphics[scale=0.25]{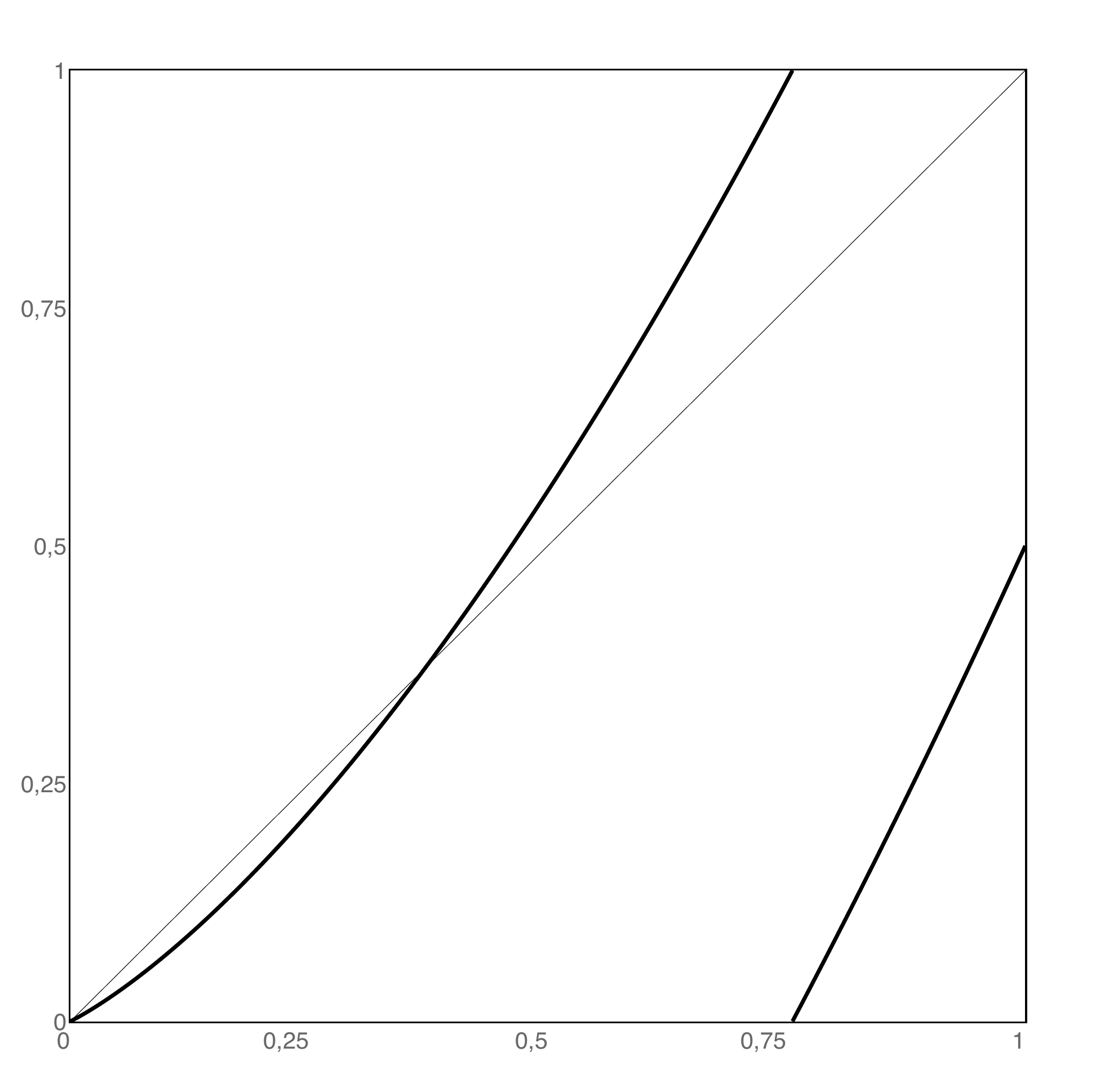} 
  \caption{}
  \label{fig:sfig1}
\end{subfigure}
  \begin{subfigure}{.5\textwidth}
  \centering
  \includegraphics[scale=0.26]{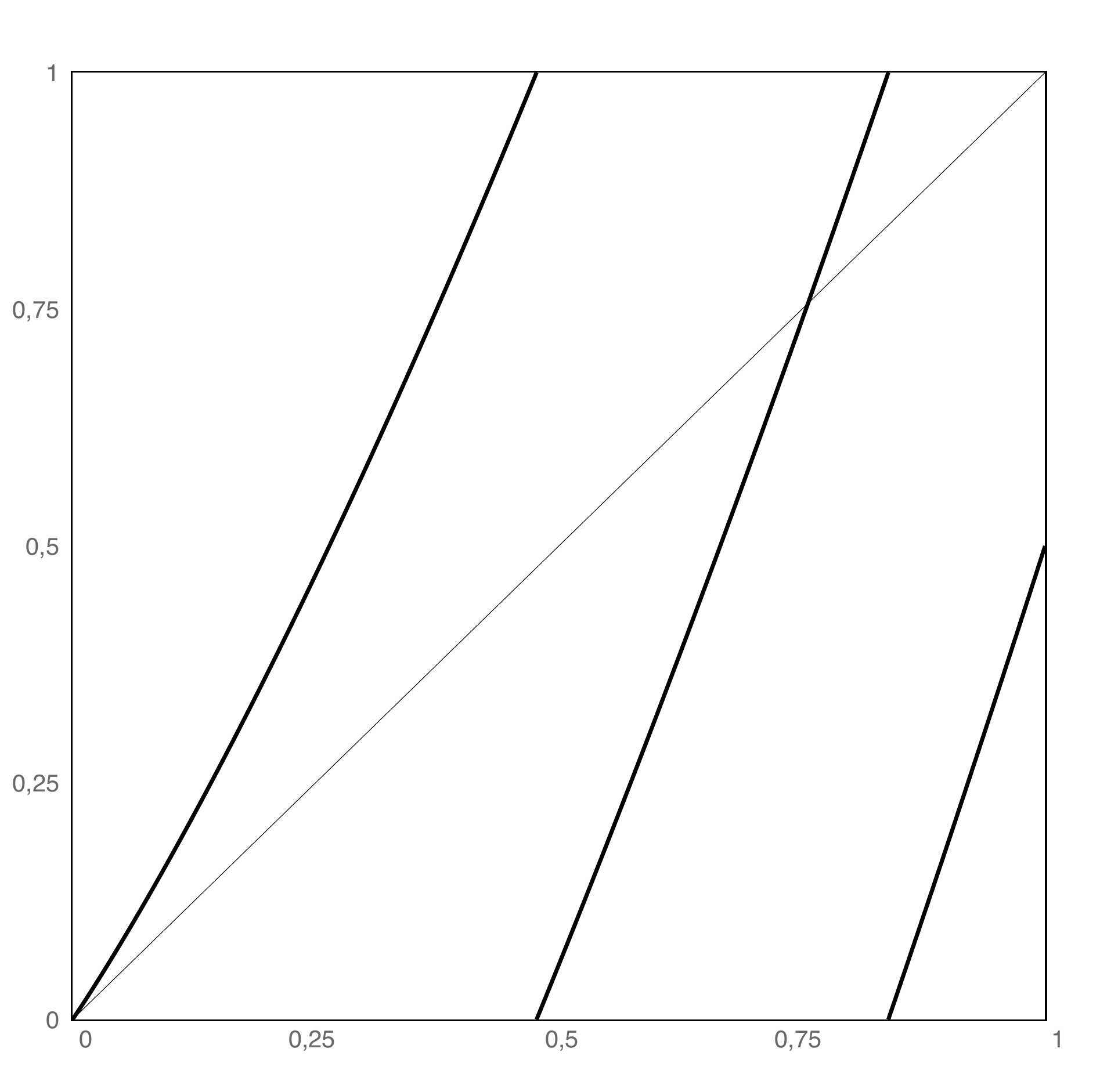} 
  \caption{}
  \label{fig:sfig2}
\end{subfigure}
   \caption{Perturbed Pomeau-Manneville maps of the circle (drawn on [0,1] modulus 1). (a) Graph of $f_{-\epsilon}$. 0 is an attracting fixed point. (b) Graph of $f_\epsilon$. The map is uniformly expanding.}
   \label{Fig2}
\end{figure}

 In other words, for a.e.  sequence $\bo \gamma=(\gamma_1,\gamma_2,\dots)$, chosen so that the $\gamma_i$ are i.i.d. uniform random variables in 
 an interval $[-\epsilon,\epsilon]$, the pushforward $f^n_{\bo \gamma}(\mu)$ converges to an absolutely continuous measure $\mu_\epsilon$
 and as $\epsilon>0$ tends to $0$  the density of this measure converges in the $L^1$ sense to the the density of $\mu$.
 Here it is crucial that the $\gamma_i$ are chosen i.i.d.  This is obvious because the pushforward of the Lebesgue measure $\lambda$
 under iterates of the map  $f_{-\epsilon}$ converge to the dirac measure at the attracting fixed point of $f_{-\epsilon}$.

One also can construct sequences $\bo \gamma$ so that $(f^n_{\bo \gamma})_*(\lambda)$ accumulates both 
to singular as well as to absolutely continuous invariant measures.   Indeed, given a sequence of integers
 $0=k_0<k_1<k_2<\dots$ define $\bo \gamma$ so that for all integers $i>0$ and $j\ge 0$
$$
\gamma_i= \left\{ \begin{array}{rl}  \epsilon &\mbox{ for }k_{2j}\,\,\,\,\,\,  < i \le  k_{2j+1}\\
-\epsilon &\mbox{ for }k_{2j+1}< i \le  k_{2j+2} 
\end{array}\right.$$

So $f^n_{\bo \gamma}$  is a composition of the expanding maps $f_\epsilon$ (the 2nd iterate of this map is expanding) 
and the $f_{-\epsilon}$ which has an attracting fixed point, see Fig. \ref{Fig2}. 
Provided we choose the sequence so that $k_{i+1}-k_i$ is sufficiently large compared to $k_i$ 
the sequence of measures $(f^n_{\bo \gamma})_*\lambda$ does not stay close to the absolutely continuous
invariant measure of $f$. Indeed, for a suitable choice of the sequence $k_i$, for $n=k_{2i}\to \infty$ the measure
$(f^n_{\bo \gamma})_*\lambda$ converges to the dirac measure at the fixed point at the attracting fixed point for $f_{\gamma_1}$,
while for $n=k_{2i+1}\to \infty$ the measure $(f^n_{\bo \gamma})_*\lambda$ converges to the absolutely continuous invariant measure
of $f_{\epsilon}$ (which is close to the absolutely continuous invariant measure of $f$ when $\epsilon>0$ is small).

{\bf Acknowledgments:} The authors thank G. Keller, C. Liverani, V. Baladi and D. Turaev for fruitful conversations related to this topic. This work was partially support by by the European Union's Seventh Framework Programme, ERC AdG Grant No. 339523 RGDD (SvS), FAPESP project 15/08958-4 (TP). We also acknowledge partial support by EU Marie-Curie IRSES Brazilian-European partnership in Dynamical Systems (FP7-PEOPLE-2012-IRSES 318999 BREUDS).

\section{Proof}\label{sec:prfthmB}

In this section we will consider piecewise expanding maps, and put ourselves in the setting of \cite{Saus}.

\subsubsection{Assumptions on the maps}\label{assumptions:maps}
 Suppose $\Omega\subset \R^N$ is a compact set with $\Omega\subset \clos(\inter\Omega)$, and $\Gamma$ is  a metric space that will serve as indexing set for the perturbations. Consider a collection of maps $\{F_\gamma\}_{\gamma\in\Gamma}$. $F_{\gamma}: \Omega\rightarrow\Omega$ for which exists $k\in \N$, $\gamma$ dependent partitions $\{U_{\gamma}^{(i)}\}_{1\leq i\leq k}$ of $\Omega$,  and  neighbourhoods $V^{(i)}$ of  $U_{\gamma}^{(i)}$ for $i=1,\dots,k$ and $\gamma\in \Gamma$, and maps $F_{\gamma}^{(i)}:V^{(i)} \rightarrow \Omega$ such that for every $\gamma\in\Gamma$
\begin{itemize}
\item[(ME1)]	${F_{\gamma}}|_{{U_{\gamma}^{(i)}}}={F_{\gamma}^{(i)}}|_{U_{\gamma}^{(i)}}$, and $B_{\epsilon_0}(F_{\gamma}(U_{\gamma}^{(i)}))\subset F_{\gamma}^{(i)}(V^{(i)})$ $\forall i=1,...,k$
\item[(ME2)]	$F_{\gamma}^{(i)}$ is $C^{1+\alpha}$ diffeomorphism meaning that $F_{\gamma}^{(i)}$ is a $C^1$ diffeomorphism, and the Jacobian  is uniformly H\"older, so for all $\epsilon<\epsilon_0$:
$$
\left|\det D_x {F_{\gamma}^{(i)}}^{-1}-\det D_y{ F_{\gamma}^{(i)}}^{-1}\right|\leq c\left |\det D_z{F_{\gamma}^{(i)}}^{-1} \right| \epsilon^\alpha, \quad x,y\in B_\epsilon(z)\cap F_{\gamma}^{(i)}(V^{(i)})
$$
\item[(ME3)] 	$m(\Omega \backslash \cup_i U_{\gamma}^{(i)})=0$

\item[(ME4)] the map $F_\gamma$ is expanding: exists $ s_\gamma \in(0,1)$ such that $\forall i\in\{1,...,k\}$ $\forall u,v\in F_{\gamma}^{(i)}(V^{(i)})$ with $d(u,v)<\epsilon_0$,  $d({F_{\gamma}^{(i)}}^{-1}(u),{F_{\gamma}^{(i)}}^{-1}(v))< s_\gamma d(u,v)$
\item[(ME5)]
$$
G^{(\gamma)}(x,\epsilon):=\sum_i\frac{m({F_{\gamma}^{(i)}}^{-1}(B_\epsilon(\partial F_{\gamma}^{(i)}U_{\gamma}^{(i)}))\cap B_{(1- s_\gamma)\epsilon})}{m(B_{(1-s_\gamma)\epsilon}(x))}
$$
$$
G^{(\gamma)}(\epsilon):=\sup_xG^{(\gamma)}(\epsilon,x)
$$
then 
$$
\sup_{\delta\leq\epsilon_0}\left[s_\gamma^\alpha+2\sup_{\epsilon\leq\delta}\frac{G^{(\gamma)}(\epsilon)}{\epsilon^\alpha}\delta^\alpha\right]<\rho<1
$$
where $\rho$ does not depend on $\gamma\in\Gamma$.
\item[(ME6)] there is $\hg\in\Gamma$ such that the transfer operator of $F_{\hg}$ has a unique eigenfunction $\phi_{\hg}$ in $V_\alpha$, a Banach subspace of $L^1$ defined in \eqref{QuasHoldSpac} below.
\end{itemize}
\begin{remark}
For what concerns hypothesis (ME6), examples in the one-dimensional case of conditions that imply uniqueness of the eigenvalues, can be found in several references (for example \cite{ViaSdds}, \cite{LasMac}). This condition is usually implied by the uniqueness of the absolutely continuous invariant measure plus a mixing condition.
\end{remark}

\begin{remark}
Although condition (ME5) requires uniformity of the upper bound on the function describing the complexity of the partition in relation to the expansion of the maps. The condition might seem rather artificial, but in \cite{Gup} is proven that considering maps satisfying (ME1)-(ME4) if their partitions sets have piecewise $C^2$ boundaries with uniformly bounded $C^2$-norm and, these hypersurfaces are in one-to-one correspondence and have small Hausdorff distance then (ME5) is automatically satisfied whenever is satisfied by one of the maps.
\end{remark}

 We first report some preliminary results from which we derive spectral properties of the transfer operators of any function $F$ from the collection and its perturbations when they are restricted to the Banach subspace $V_\alpha\subset L^1(\R^N)$ of quasi-H\"older functions (defined in Section \ref{QuasHoldFun}). For a definition of the transfer operator and a discussion of some of its properties see Section \ref{Sec:prel}. Under assumption (ME6), for each perturbed operator, the space of quasi-H\"older functions splits into the direct sum of invariant subspaces: one corresponds to the invariant direction while the restriction of the transfer operator to the other is a contraction. We conclude by showing how the presence of such a spectral gap implies the result.

\subsection{Preliminaries}\label{prelB}
 
 In \cite{Gup}, the authors prove memory loss looking at the action of the maps on an invariant cone of functions by generalising to multidimensional maps a construction introduced in \cite{Liv5} for the one-dimensional case. This procedure is the analogous in the piecewise case of what we present for $C^{1+\nu}$ maps in the appendix, but it requires that all the maps have some mixing property on the whole phase space which is always the case in the regular case, while it has to be assumed as an extra hypothesis in the piecewise case.  Our argument allow us to drop this hypothesis, although it works only when considering composition of small perturbations of a given map.

We exploit a well known property exhibited by the transfer operator of some maps: quasi-compactness (\cite{BalYou,Kel,Bal,KellPro,Liv5,You1,You2}). For a definition of the transfer operator and a discussion of some of its properties see Section \ref{Sec:TranOper} in the appendix.
 \subsubsection{Stochastic Stability: Stationary Case}\label{SecSmoStatStab}

Now fix a Borel probability measure $\nu$ on the metric space $\Gamma$ (endowed with the $\sigma$-algebra of Borel sets), and consider the asymptotic behaviour of the random trajectories $\{x_i\}_{i\in\N}$ with 
$$
x_i:=F_{\gamma_i}\circ...\circ F_{\gamma_1}(x_0)\quad \forall i\in\N
$$
where the $\{\gamma_i\}_{i\in\N}$ are sampled independently from $\Gamma$ with distribution given by $\nu$. This random dynamical system is equivalent to the skew product dynamical system
\begin{align*}
\mc F: M\times \Gamma^\N&\rightarrow M\times \Gamma^\N\\
(x,\bo \gamma)&\rightarrow (f_{\gamma_1}(x),\sigma(\bo \gamma))
\end{align*}
on $(M\times \Gamma^\N,m\otimes \nu^{\otimes\N})$, where $\bo \gamma\in \Gamma^\N$, and $\sigma$ is the left-sided shift. 
Denoting with $\mc L_\gamma$ the transfer operator associated to $F_\gamma$, one can define the \textit{average transfer operator} 
$$
\hat{\mc L}_\nu \phi:=\int_\Gamma \mc L_\gamma\phi d\nu(\gamma).
$$
The density $\phi_\nu$ satisfying $\mc L_\nu\phi_\nu=\phi_\nu$ is called stationary density. Whenever $\nu=\delta_\gamma$ for some $\Gamma$, $\hat{ \mc L}_\nu=\mc L_\gamma$ and $\phi_\nu$ (if it exists) is the invariant density under the map $F_\gamma$.

To prove Part (i) of Theorem \ref{thmA}, we will use the spectral properties of $\mc L_{\hg}$ and $\mc L_\nu$ in a way similar to what has already been shown for one-dimensional piecewise maps in \cite{ViaSdds}.
\subsubsection{Spectral Theorems for the Transfer Operators}
It is often very useful to restrict the action of the transfer operator to some Banach space contained in $L^1$. It has been shown in many cases how such a restriction have nice spectral properties that imply, among others, existence of invariant absolutely continuous measures and exponential mixing of correlations between observables. Suppose that $V\subset L^1(M)$, and  $(V,\|\cdot\|_V)$ is a Banch space. A  theorem by Ionescu-Tulcea and Marinescu give a criterion to establish the spectral properties of $\mc L$. We report here this theorem in the case where the ambient space is $L^1(M)$

\begin{theorem}[\cite{IonTul}]\label{ThmIonTulcMar}
Let $(V,\|\cdot\|_V)$ be a Banach closed subspace of $L^1(M)$ such that if $\{\phi_n\}_{n\in\N}\subset V$, $\|\phi_n\|_V\leq K$ is such that $\phi_n\rightarrow \phi$ in $L^1$, then $\phi\in V$, and $\|\phi\|_V\leq K$. Let $\mc C(V)$ be the class of linear bounded operators with image in $V$ satisfying
\begin{itemize}
\item[(1)] there exists $H$ s.t. $|\mc P^n|_V\leq  H$  $\forall n\in\N$
\item [(2)]there exists $0<r<1$ and $R>0$ such that 
\begin{equation}\label{GenLasYor}
\|\mc P\phi\|_V\leq r\|\phi\|_V+R\|\phi\|_1
\end{equation}
\item[(3)]  $\mc P(B)$ is compact in $L^1$ for every bounded $B$ in $(V,\|\cdot\|_V)$.
\end{itemize}
Then every $\mc P\in \mc C(V)$ has only a finite number of eigenvalues $\{c_1,...,c_p\}$ of modulus 1 with finite dimensional eigenspaces $\{X_1,...,X_p\}$, and 
$$
\mc P=\sum_{i=1}^pc_i  P_i+ P_0
$$
where, if $\{\pi^{(i)}\}_{i=\{1,...p\}}$, $\pi^{(0)}$ are projections relative to the splitting 
$$
V=\bigoplus_{i=1}^pX_i\oplus X_0
$$
$P_i:=\mc P\circ \pi^{(i)}$, and $\|P_0^n\|_V=O(q^n)$ with $q\in(0,1)$.
\end{theorem}

The theorem can be used to understand the behaviour of the transfer operator for a variety of maps. Most of the requirements are automatically satisfied by the transfer operator, and the only thing that requires an additional proof is inequality \eqref{GenLasYor} often referred as a Lasota-Yorke type of inequality. For such an inequality to hold, the Banach space $(V,\|\cdot\|_V)$ must be chosen carefully.

In the following we need a result that deals with perturbed transfer operators. This is treated in various references and presented in different formulations. Among others we cite \cite{KelLiv, KellPro, ViaSdds, Bal}. We report the statement that can be found in \cite{ViaSdds} for transfer operators associated to piecewise-expanding maps, and that can be generalised without any extra effort to the above  setting.

\begin{theorem}[\cite{ViaSdds}]\label{PertOper}
Suppose $(V,\|\cdot\|_V)$ is a closed Banach space which is a subspace of $L^1(M)$. Let $C>0$, $q<1$, $\lambda<1$, and $ \mc P_\epsilon:V\rightarrow V$, be a family of linear operators satisfying:
\begin{itemize}
\item $\int  \mc P_\epsilon \phi dm=\int \phi dm$, and $\phi\geq 0$ implies $ \mc P_\epsilon \phi\geq 0$;
\item $\| \mc P^n_\epsilon \phi\|_V\leq C\lambda^n\|\phi\|_V+C\|\phi\|_1$;
\end{itemize}
for every $n\geq 1$, $\epsilon\geq 0$, and $\phi\in V$. Suppose that

\begin{itemize}
\item for $n\geq 1$ there is $\epsilon(n)$ so that for all $\phi\in V$ and all $\epsilon\in(0,\epsilon(n))$
\begin{equation}\label{AppBound}
\| \mc P_0^n\phi- \mc P^n_\epsilon \phi\|_1\leq C\lambda^n\|\phi\|_V
\end{equation}
\item $\spec(\mc P_0)=\{1\}\cup\Sigma_0$, where 1 is a simple eigenvalue and $\Sigma_0\subset\{z\in\C: |z|\leq q\}$.
\end{itemize}

Fix $\tilde q\in(\max\{\sqrt{q},\sqrt \lambda\},1)$. Then, for any small enough $\epsilon>0$, $\spec( \mc P_\epsilon)=\{1\}\cup \Sigma_\epsilon$, where 1 is a simple eigenvalue and $\Sigma_\epsilon\subset\{z\in\C:|z|\leq\tilde q\}$.

\end{theorem}

The above theorem states that, under some hypotheses, when dealing with a quasi-compact transfer operator with 1 as unique simple eigenvalue,  small perturbations do not jeopardise quasi-compactness.

\subsubsection{Quasi-H\"older spaces $V_{\alpha}$ }\label{QuasHoldFun}

In this section we report the definition of quasi-H\"older space as presented in \cite{Saus}. This is the Banach space on which we restrict the action of the Perron-Frobenius operator associated to $F$.
Given $\phi\in L^1(\R^N)$ and $S$ a Borel subset of $\R^N$, define
$$
\osc(\phi, S):=\essup_S\phi-\esinf_S\phi.
$$

For all $\epsilon>0$ and $\phi\in L^1(\R^N)$, the map $x\mapsto \osc(\phi,B_\epsilon(x))$ is measurable (in particular is lower semi-continuous). Given $\alpha\in(0,1)$, $|f|_\alpha$ is defined (finite or infinite) as
$$
|\phi|_\alpha:=\sup_{0<\epsilon\leq\epsilon_0} \epsilon^{-\alpha}\int_{\R^N}\osc(\phi,B_\epsilon(x))dm(x),
$$
and  $V_\alpha$ is
\begin{equation}\label{QuasHoldSpac}
V_\alpha:=\{\phi\in L^1(\R^N):\quad |\phi|_\alpha<\infty\}.
\end{equation}

The space $V_\alpha$, endowed with the norm $\|\cdot\|_\alpha:=|\cdot|_\alpha+\|\cdot\|_1$, where $\|\cdot\|_1$ is the $L^1$ norm, is a Banach space.
 
\subsubsection{Spectral Properties of $\mc L$ on $V_{\alpha}$}

The transfer operator $\mc L$ associated to $F$ satisfying (ME1)-(ME5), fulfils a Lasota-Yorke type of inequality. 
\begin{proposition}[\cite{Saus}]\label{LasYorMultPiecIneq}
Suppose $F$ satisfies (ME1)-(ME5). If $\epsilon_0$ is small enough, there exists $\eta\in(0,1)$ and $C<0$ such that for all $\phi\in V_\alpha$ 
$$
\mc L\phi\in V_\alpha\mbox{ and }|\mc L \phi|_\alpha\leq \eta|\phi|_\alpha+C\int_{\R^N}|\phi| dm.
$$
\end{proposition}

Theorem \ref{ThmIonTulcMar} by Ionescu-Tulcea and Marinescu gives the spectral properties of $\mc L$. Whenever the transfer operator $\mc L$ has $\{1\}$ as unique eigenvalue which is also simple, we obtain the following splitting

\begin{proposition}
Let $F$ be a map that satisfies (ME1)-(ME6) then:
$$
\spec(\mc L)=\{1\}\cup \Sigma_0
$$
where 1 is a simple eigenvalue and $\Sigma_0$ is a disc of radius $q<1$, and 
$$
V_\alpha=\R \phi_0\oplus X_0.
$$
\end{proposition}

\subsection{Perturbations}\label{PertPiecCaseSec}

In \cite{Cow}, the author treats the problem of stochastic stability for the invariant density for multidimensional piecewise expanding maps with piecewise smooth boundaries of the partitions. We address the same problem in the setting presented above, which makes it natural to consider perturbations of the system $F:\Omega\rightarrow\Omega$ with different regularity partitions as long as the branches admit an extension to the same neighbourhood. 

\subsubsection{Continuity assumptions}\label{assumptions:continuity}

 Given any $\hg\in\Gamma$, we say that the collection $\{F_{\gamma}\}_{\gamma\in\Gamma}$  is continuous at $\hg$ if for every $\epsilon>0$ there is a $\delta>0$ such that the $\delta$-ball $B_\delta(\hg)$ has the following properties:
\begin{itemize}
\item[(CM1)]
for each $\gamma_1,\gamma_2\in B_\delta(\hg)$, 
the $C^1$ distance between $F^{(i)}_{\gamma_1}$ and  $F^{(i)}_{\gamma_2}$ is at most $\epsilon$;
\item[(CM2)]for every $\gamma\in B_{\delta}(\hg)$ and $1\leq i \leq k$, $m(U_\gamma^{(i)}\triangle U^{(i)})\leq\epsilon$, where $\triangle$ stands for the symmetric difference.
\end{itemize}

\subsubsection{Notation}

Define the set of multi-indices $\mc I_n:=\{1,...,k\}^n$. For $\bo i\in \mc I_n$ and $\bo \gamma=(\gamma_1,...,\gamma_n)\in\Gamma^n$  call 
$$
F^n_{\bo \gamma}(x):=F_{\gamma_n}... F_{\gamma_1}(x) \quad\mbox{and}\quad F^{(\bo i)}_{\bo \gamma}(x):=F^{(i_n)}_{\gamma_n}... F^{i_1}_{\gamma_1}(x), 
$$ 
whenever is well defined.
For $\bo \gamma\in\Gamma^n$, define $\mc L^n_{\bo \gamma}:=\mc L_{\gamma_n}...\mc L_{\gamma_1}$.
For all $n\in\N$ and $\bo \gamma\in\Gamma^n$, let us denote with $\{U^{(\bo i)}_{\bo \gamma}\}_{\bo i \in \mc I_n}$ the partition of $\Omega$ modulo a negligible subset such that, if $x\in U^{(\bo i)}_{\bo \gamma}$, then $x\in U^{i_1}_{\gamma_1} $, and $F^{(i_j)}_{\gamma_j}... F^{(i_1)}_{\gamma_1}(x)\in U^{(i_{j+1})}_{\gamma_{j+1}}$ for $1\leq j<n$.  Notice that $F^{(\bo i)}_{\bo \gamma}$ is well defined on  $U^{(\bo i)}_{\bo \gamma}$, and its restriction equals $F_{\bo \gamma}$.

We fix $\hg\in\Gamma$. From now on, $F_{\hg}$ represent the 'unperturbed' map, and will be denoted as $F$ and $\mc L_{\hg}$ will be denoted as $\mc L$.
\begin{remark}
Some of the $U^{(\bo i)}_{\bo \gamma}$ might be empty, or of measure zero.
\end{remark}

\subsubsection{Perturbation Results}

We now prove that if one randomly compose sufficiently small perturbations, then the average transfer operator of the perturbed map satisfies a uniform Lasota-Yorke type of inequality.  
\begin{lemma}\label{PErtLasYorPie}
Let $\{F_{\gamma}\}_{\gamma\in\Gamma}$ be perturbations of $F:=F_\hg$ as in (ME1)-(ME6) continuous at $\hg$ as in (CM1)-(CM2). Then, there exists $C>0$ and $\tilde\eta\in(0,1)$, and $\Gamma'\subset\Gamma$ neighbourhood of $\hg$ such that, for every probability measure $\nu$ with $\supp\nu\subset \Gamma'$ and all $\phi\in V_\alpha$
$$
|\hat{\mc L_\nu} \phi|_\alpha\leq \tilde \eta |\phi|_\alpha+C\int|\phi|dm
$$
\end{lemma}
\begin{proof}
The proof of this lemma follows from Proposition \ref{LasYorMultPiecIneq}. Since $F_{\gamma}$ satisfies (ME1)-(ME5), Proposition \ref{LasYorMultPiecIneq} implies that it satisfies a Lasota-Yorke inequality with $\tilde\eta(\gamma)\in(0,1)$, and $C(\gamma)>0$. As proved in \cite{Saus}, $C(\gamma)$ has a uniform bound for every $\gamma\in\Gamma$ and
$$
\eta(\gamma)=(1+cs_\gamma^\alpha\epsilon_0^\alpha)\rho.
$$  
Choosing $\epsilon_0$ so that $\eta(\hg)<1$, since $s_\gamma\rightarrow s_{\hg}$ for $\gamma\rightarrow\hg$ by (CM1), one can pick $\Gamma'\subset\Gamma$ a neighbourhood of $\hg$ such that $\eta(\gamma)<\tilde\eta<1$ for all $\gamma\in\Gamma'$. This implies that $\forall \gamma\in\Gamma'$
$$
|\mc L_\gamma \phi|_\alpha\leq \tilde \eta |\phi|_\alpha+C\int|\phi|dm.
$$
Now,
\begin{align*}
|\hat{\mc L}_\nu\phi|_\alpha&=\left|\int_{\Gamma}\mc L_\gamma\phi d\nu(\gamma)\right|_\alpha=\sup_{0<\epsilon\leq\epsilon_0}\epsilon^{-\alpha}\int_{\R^N}\osc\left(\int_{\Gamma}\mc L_\gamma\phi d\nu(\gamma),B_\epsilon(x)\right)dm(x)\\
\end{align*}
and from the definition of oscillation, since $\essup$ and $\esinf$ are respectively a convex and concave function on the essentially bounded functions,
\begin{align*}
|\hat{\mc L}_\nu\phi|_\alpha&\leq \sup_{0<\epsilon\leq\epsilon_0}\epsilon^{-\alpha}\int_{R^N}\int_{\Gamma} \osc\left(\mc L_\gamma\phi,B_\epsilon(x)\right)d\nu(\gamma)dm(x)\\
&\leq  \sup_{0<\epsilon\leq\epsilon_0}\int_{\Gamma}\epsilon^{-\alpha}\int_{R^N} \osc\left(\mc L_\gamma\phi,B_\epsilon(x)\right)dm(x)d\nu(\gamma)\\
&\leq \int_\Gamma|\mc L_\gamma\phi|_\alpha d\nu(\gamma)\\
&\leq \tilde \eta |\phi|_\alpha+C\int|\phi|dm.
\end{align*}

\end{proof}
Since $\tilde \eta<1$ one immediately gets 
\begin{lemma} \label{BounNormPertActMult}
\begin{itemize}
\item[(1)]\label{PertLasYorMult}
for all $n\in\N$ and $\bo \gamma\in (\Gamma')^n$ 
$$
|\mc L_{\gamma_n}...\mc L_{\gamma_1}\phi|_\alpha\leq\tilde\eta^{n}|\phi|_\alpha+\frac{C}{1-\tilde \eta}\int|\phi|dm
$$
which is thus uniformly bounded on $n$,
\item[(2)] for all $\epsilon>0$ and all densities $\phi\in V_\alpha$, there is a $\tilde n(\phi,\epsilon)\in\N$ such that 
$$
|\mc L_{\gamma_n}...\mc L_{\gamma_1}\phi|_\alpha\leq \frac{C}{1-\tilde \eta}+\epsilon
$$
for all $n>\tilde n(\phi,\epsilon)$.
\end{itemize}
\end{lemma}

We now prove a perturbation estimate crucial in determining  the spectral properties of the perturbed transfer operators using similar estimates to the procedure in \cite{ViaSdds}.
\begin{proposition}\label{MultiPertRes}
Let $\{F_{\gamma}\}_{\gamma\in\Gamma}$ be perturbations of $F:=F_\hg$ as in (ME1)-(ME6) continuous at $\hg$ as in (CM1)-(CM2), then there exists $0<\tilde s\leq 1$, and $\tilde C>0$ such that for all $n\in \N$ there is $\delta>0$, satisfying
\begin{equation}\label{ExpUppBoun}
\|\mc L^n_{\bo \gamma}\phi-\mc L^n \phi\|_1\leq \tilde C\tilde s^n\|\phi\|_\alpha,\quad \forall \phi\in V_\alpha
\end{equation}
for all $\bo \gamma\in B_\delta(\hg)^n$.
\end{proposition}

\begin{proof}
\small
\begin{align*}
\int|\mc L^n_{\bo \gamma}\phi(x)&-\mc L^n \phi(x)|dm(x) \leq\\
 &\leq  \sum_{\bo i\in\mc I_n}\left[\quad \int_{F^{(\bo i)}_{\bo \gamma}(U^{(\bo i)}_{\bo \gamma})\cap F^{(\bo i)}(U^{(\bo i)})}\left|(\phi|\det {DF^{(\bo i)}_{\bo \gamma}}|^{-1})\circ (F^{(\bo i)}_{\bo \gamma})^{-1}(x)-(\phi|\det {DF^{(\bo i)}}|^{-1})\circ{ F^{(\bo i)}}^{-1}(x)\right| dm(x)+\right. \\
&\quad+ \left. \int_{F^{(\bo i)}_{\bo \gamma}(U^{(\bo i)}_{\bo \gamma})\backslash F^{(\bo i)}(U^{(\bo i)})}  \left|(\phi|\det {DF_{\bo \gamma}^{(\bo i)}}|^{-1})\circ (F^{(\bo i)}_{\bo \gamma})^{-1}(x)\right| dm(x)+ \right. \\
&\left. \quad+\int_{F^{(\bo i)}(U^{(\bo i)})\backslash F^{(\bo i)}_{\bo \gamma}(U^{(\bo i)}_{\bo \gamma})}\left|(\phi|\det {DF^{(\bo i)}}|^{-1})\circ (F^{(\bo i)})^{-1}(x)\right| dm(x)\quad \right] \\
&=:\sum_{\bo i\in\mc I_n}[(A)_{\bo i}+(B)_{\bo i}+(C)_{\bo i}]
\end{align*}
\normalsize
We first treat $(A)_{\bo i}$ and then $(B)_{\bo i}$ with $(C)_{\bo i}$ for which analogous arguments hold. Notice that
\begin{align*}
&\phantom{\leq}\left|(\phi|\det {DF^{(\bo i)}_{\bo \gamma}}|^{-1})\circ (F^{(\bo i)}_{\bo \gamma})^{-1}-(\phi|\det DF^{(\bo i)}|^{-1})\circ {F^{(\bo i)}}^{-1}\right| \leq\\
&\leq\left|\phi\circ (F^{(\bo i)}_{\bo \gamma})^{-1}-\phi\circ {F^{(\bo i)}}^{-1}\right| \left|\det {DF^{(\bo i)}}^{-1}\right| +\left|\phi\circ (F^{(\bo i)}_{\bo \gamma})^{-1}\right|\left|\det {DF^{(\bo i)}_{\bo \gamma}}^{-1}-\det {DF^{(\bo i)}}^{-1}\right|
\end{align*}
which upper bounds $(A)_i$ with the sum of two terms. The first one is
\begin{align*}
(1)_i&:= \int_{F_{\bo \gamma}^{(\bo i)}(U_{\bo \gamma}^{(\bo i)})\cap F^{(\bo i)}(U^{(\bo i)})} \left|\phi\circ (F^{(\bo i)}_{\bo \gamma})^{-1}(x)-\phi\circ {F^{(\bo i)}}^{-1}(x)\right| \left|\det {DF^{(\bo i)}}^{-1}\right|dm(x)
\end{align*}
and one can chose a suitable $\delta>0$ such that $\forall\gamma\in B_\delta(\hg)$ and all $i\in\{1,...,k\}$ 
$$|{F^{(i)}_\gamma}^{-1}-{F^{(i)}}^{-1}|<\xi(\delta)$$ with $\xi(\delta)<\epsilon_0(1-s)$, so by induction, 
\small
\begin{align*}
|{F^{(\bo i)}_{\bo \gamma}}^{-1}(x)-{F^{(\bo i)}}^{-1}(x)|&\leq |{ F_{\gamma_1}^{(i_1)}}^{-1}\circ{F^{(i_2,..,i_n)}_{(\gamma_2,...,\gamma_n)}}^{-1}(x)  -{F^{(i_1)}}^{-1}\circ {F^{(i_2,..,i_n)}_{(\gamma_2,...,\gamma_n)}}^{-1}(x) |+ \\
& ~ |{F^{(i_1)}}^{-1}\circ {F^{(i_2,..,i_n)}_{(\gamma_2,...,\gamma_n)}}^{-1}(x)-{F^{(i_1)}}^{-1}\circ {F^{(i_2,...,i_n)}}^{-1}(x)|\\
&\leq \xi(\delta)+ s|{F^{(i_2,..,i_n)}_{(\gamma_2,...,\gamma_n)}}^{-1}(x)- {F^{(i_2,...,i_n)}}^{-1}(x)|
\end{align*}
\normalsize
for all $x\in{F^{(\bo i)}_{\bo \gamma}(U^{(\bo i)}_{\bo \gamma})\cap F^{(\bo i)}(U^{(\bo i)})}$, yielding 

$$
|{F^{(\bo i)}_{\bo \gamma}}^{-1}-{F^{(\bo i)}}^{-1}|\leq \xi(\delta)\frac{1}{1-s}.
$$
This implies that, for any fixed $\epsilon$, choosing $\delta$ so that $\xi(\delta)/(1-s)<\epsilon$
\begin{align*}
(1)_i&\leq \int_{F^{(\bo i)}(U^{(\bo i)})} \osc(\phi,B_\epsilon({F^{(\bo i) }}^{-1}(x))) |\det {DF^{(\bo i)}}^{-1}|dm(x)\\ 
&\leq \int_{U^{(\bo i)}}\osc(\phi,B_\epsilon(y))dm(y)\\ 
\end{align*}
Taking the sum over $\mc I_n$
\begin{align*}
\sum_{\bo i\in\mc I_n}\int_{U^{(\bo i)}}\osc(\phi,B_\epsilon(y))dm(y)&\leq\int_{\R^N}\osc(\phi,B_\epsilon(y))dm(y)\\
&\leq \epsilon^{\alpha}|\phi|_\alpha
\end{align*}

For the second term
\begin{align*}
(2)_i:=&\int_{F^{(\bo i)}_{\bo \gamma}(U^{(\bo i)}_{\bo \gamma})\cap F^{(\bo i)}(U^{(\bo i)})}\left|\phi\circ (F^{(\bo i)}_{\bo \gamma})^{-1}\right|\left|\det {DF_{\bo \gamma}^{(\bo i)}}^{-1}-\det {DF^{(\bo i)}}^{-1}\right|dm\leq \\
&\leq m(\Omega)  \xi(\delta) \essup_\Omega |\phi|	\\
\end{align*}
where $\xi(\delta)$ is a number that can be made arbitrarily small restricting $\delta$. 
From compactness of $\Omega$, there exists a $\bar x$ such that
\begin{align}
\essup_\Omega|\phi|&=\essup_{B_{\frac{\epsilon_0}{2}}(\bar x)}|\phi|\nonumber\\
&\leq\frac{1}{m(B_{\frac{\epsilon_0}{2}}(x))}\int_{B_{\frac{\epsilon_0}{2}}(x)}[|\phi(y)|+\osc(\phi,B_{\frac{\epsilon_0}{2}}(y))]dm(y) \label{essupineq1}
\end{align}
one obtains
$$
(2)_i\leq C' \xi(\delta)\|\phi\|_\alpha
$$
from which
\begin{align*}
\sum_{\bo i\in\mc I_n}(A)_{\bo i}&\leq\epsilon^{\alpha}|\phi|_\alpha+C'(\#\mc I_n)\xi(\delta)\|\phi\|_\alpha.
\end{align*}
For what concerns $(B)_i$
\begin{align*}
(B)_{\bo i}&= \int_{F^{(\bo i)}_{\bo \gamma}(U^{(\bo i)}_{\bo \gamma})\backslash F^{(\bo i)}(U^{(\bo i)})} |(\phi\det {DF_{\bo \gamma}^{(\bo i)}}^{-1})\circ (F^{(\bo i)}_{\bo \gamma})^{-1}(x)| dm\\\
&\leq m(F^{(\bo i)}_{\bo \gamma}(U^{(\bo i)}_{\bo \gamma})\backslash F^{(\bo i)}(U^{(\bo i)}))\tilde s^n \essup_\Omega |\phi| \\
&\leq \xi(\delta) C'' \tilde s_\gamma^n\|\phi\|_\alpha
\end{align*}
where we upper bounded $m(F^{(\bo i)}_{\bo \gamma}(U^{(\bo i)}_{\bo \gamma})\backslash F^{(\bo i)}(U^{(\bo i)}))$ with $\xi(\delta)$ that thanks to (CM2) can be made arbitrarily small reducing $\delta$. Summing all the contributions 
\begin{align*}
\sum_{\bo i\in\mc I_n}(B)_{\bo i}&\leq (\#\mc I_n)\xi(\delta)C''\tilde s_\gamma^n\|\phi\|_\alpha.
\end{align*}
The sum of the $(C)_{\bo i}$ terms can be upper bounded analogously. As already pointed out, for a smaller $\delta$ we can make the upper bound arbitrarily small. This allows, in particular, to obtain an exponential upper bound as \eqref{ExpUppBoun} with respect to some $\tilde s\in(0,1)$.
\end{proof}

We can generalise the above proposition to the case of averaged transfer operators.
\begin{proposition}
There exists $0<\tilde s\leq 1$, and $\tilde C>0$ such that for all $n\in \N$ there is $\delta>0$, satisfying
$$
\|\hat{\mc L}_\nu^n\phi-\mc L^n\phi\|_1\leq \tilde C\tilde s^n\|\phi\|_\alpha, \quad \forall\phi\in V_\alpha
$$
for every probability measure $\nu$ with $\supp\nu\subset	B_\delta(\hg)$.
\end{proposition}
\begin{proof}
\begin{align*}
\|\hat{\mc L}_\nu^n\phi-\mc L^n\phi\|_1&\leq\int \left|\int_{\Gamma^n}\mc L^n_{\bo\gamma}\phi(x)d\nu^{\otimes n}(\bo\gamma)-\mc L^n\phi(x)\right|dm(x)\\
&\leq\int\int_{\Gamma^n}\left|\mc L^n_{\bo\gamma}\phi(x)-\mc L^n\phi(x)\right| d\nu^{\otimes n}(\bo\gamma)dm(x)\\
&\leq\int_{\Gamma^{n}}\|\mc L^n_{\bo\gamma}\phi-\mc L^n\phi\|_1d\nu^{\otimes n}(\bo \gamma).
\end{align*}

Since $\nu$ is supported on $B_\delta(\hg)$, almost every sequence $\bo\gamma$ in the above integral will belong to $B_\delta(\hg)^n$ allowing a direct application of Proposition \ref{MultiPertRes}.
\end{proof}

The above proposition and Proposition \ref{PertOper} give the spectral properties for the perturbed transfer operators.
\begin{proposition}
There is a neighbourhood of $\hg$, $\Gamma''\subset\Gamma$ such that 
$$
\spec\mc L_\nu=\{1\}\cup\Sigma_0\quad \quad V_\alpha=\R \phi_\nu\oplus X_0
$$
for all probability measures $\nu$ with $\supp\nu\subset\Gamma''$, with $\Sigma_0$ inside a disk of radius $\tilde q\in(0,1)$, and $\phi_\nu$ is the unique stationary density. The projections associated to the splitting are 
$$
\pi_1^{(\nu)}\phi=\left(\int \phi dm\right)\phi_\nu,\quad \pi_0^{(\nu)}\phi=\phi-\left(\int \phi dm\right)\phi_\nu.
$$ 
\end{proposition}

\begin{remark}
As already remarked, the stationary measure associated with $\nu=\delta_\gamma$, $\gamma\in\Gamma''$, is the invariant measure for the map $F_\gamma$.
\end{remark}

This proposition is a direct consequence of Theorem \ref{PertOper}. One can easily prove that the invariant densities have uniformly bounded norms.

\begin{lemma}\label{boundedvar2}
For all probability measures $\nu$ with $\supp\nu\subset\Gamma'$
$$
|\phi_\nu |_\alpha\leq \frac{C}{1-\tilde\eta}
$$
\end{lemma}
\begin{proof}
From Lemma \ref{PErtLasYorPie}
$$
|\phi_\nu|_\alpha=|\mc L_\nu \phi_\nu|_\alpha \leq \tilde\eta|\phi_\nu|_\alpha+C
$$
which implies
$$
|\phi_\nu|_\alpha\leq\frac{C}{1-\tilde\eta}.
$$
\end{proof}

\subsection{Proof of Theorem~\ref{thmA}}\label{SecProfMaiTheMul}

We first prove Part (1) of Theorem A along the same lines of  \cite{ViaSdds} where it is proven in for one-dimensional maps. 

\begin{proof}[Proof of Part (1) of Theorem A]

By triangular inequality, for all $n\in\N$ and $\nu$ $\supp\nu\subset B_\delta(\hg)$
\begin{align}
\|\phi_\nu-\phi_\hg\|_1&\leq \|\phi_\nu-\mc L_\nu^n\phi_\hg \|_1+\|\mc L_\nu^n\phi_\hg-\phi_\hg\|_1\nonumber\\
&\leq\|\mc L_\nu^n\phi_\nu-\mc L_\nu^n\phi_\hg\|_1+\|\mc L_\nu^n\phi_\nu-\mc L^n\phi_\hg\|_1\nonumber\\
&\leq \tilde q^n\|\phi_\nu-\phi_\hg\|_\alpha+\tilde C\tilde s^n\|\phi_\hg\|_\alpha\label{PertStabInvDen}
\end{align}
where in \eqref{PertStabInvDen}, for the first term we used the spectral splitting and the fact that $\phi_\nu-\phi_\hg\in X_0$, and for the second term we used Proposition \ref{MultiPertRes}. By Lemma \ref{boundedvar2}, $\|\phi_\nu-\phi_\hg\|_\alpha$ is uniformly bounded for $\gamma\in\Gamma'$, and this implies the result choosing $n$ sufficiently large and adjusting $\delta>0$ accordingly.
\end{proof}

We now prove Part (2) of Theorem A. In the proof, we denote by $\phi_\gamma\in V_\alpha$, the unique invariant probability density for the map $F_{\gamma}$, where $\gamma\in\Gamma'$. 
\begin{proof}[Proof of Part (2) of Theorem \ref{thmA}]

We can restate the theorem in terms of the action of the transfer operators on the density of the initial mass distribution. We shall then prove that for all densities $\phi\in V_\alpha$ and every  $\epsilon>0$, there exists $\bar n:=\bar n(\epsilon,\phi)\in\N$ and $\delta$ independent of $\phi$ such that, for all $n\geq \bar n$ and for all sequences $\bo \gamma \in B_\delta(\hg)^n$
$$
\|\mc L_{\bo \gamma}^n\phi-\phi_\hg\|_1\leq\epsilon.
$$

Let us consider the application of the transfer operators on their arguments split by projections $\pi_1^{(\gamma)}$ and $\pi^{(\gamma)}_0$. For example:
$$
\mc L_{\gamma_n}...\mc L_{\gamma_1}\phi=\mc L_{\gamma_n}\pi^{(\gamma_n)}_0\left(\mc L_{\gamma_{n-1}}...\mc L_{\gamma_1}\phi\right)+\mc L_{\gamma_n}\pi^{(\gamma_n)}_1\left(\mc L_{\gamma_{n-1}}...\mc L_{\gamma_1}\phi\right)
$$

By induction
$$
\mc L_{\gamma_n}\mc L_{\gamma_{n-1}}...\mc L_{\gamma_1}\phi=\sum_{(i_1,...,i_n)\in\{0,1\}^n}\mc L_{\gamma_n}\pi^{(\gamma_n)}_{i_n}\mc L_{\gamma_{n-1}}\pi^{(\gamma_{n-1})}_{i_{n-1}}...\mc L_{\gamma_1}\circ\pi^{(\gamma_1)}_{i_1}\phi
$$

Since $\pi_0^{(\gamma)}$ projects on the $X_0$ space, which is invariant under $\mc L_{\gamma}$,  in the above sum only $n$ terms give a nonzero contribution: after projecting on $X_0$ the action of the operators does not leave this space and if we later project on any of the $\R \phi_{\gamma'}$ we obtain zero. This implies that only the non-increasing sequences of $\{0,1\}^n$ may correspond to non-vanishing terms:
$$
\mc L_{\gamma_n}\mc L_{\gamma_{n-1}}...\mc L_{\gamma_1}\phi=\sum_{\scriptsize\begin{array}{c}(i_1,...,i_n)\in\{0,1\}^n\\i_j\geq i_{j+1}\end{array}\normalsize}\mc L_{\gamma_n}\pi^{(\gamma_n)}_{i_n}\mc L_{\gamma_{n-1}}\pi^{(\gamma_{n-1})}_{i_{n-1}}...\mc L_{\gamma_1}\pi^{(\gamma_1)}_{i_1}\phi
$$
In the above sum:
\begin{itemize} 
\item the term with $(0,...,0,1,...,1)$ ($i_j=0$ and $i_{j-1}=1$) equals
$$
\mc L_{\gamma_n}...\mc L_{\gamma_{j}}(\phi_{\gamma_{j-1}}-\phi_{\gamma_j}),
$$
\item the term (0,...,0) equals $\phi_{\gamma_n}$,
\item the term (1,...,1) equals
$$
\mc L_{\gamma_n}...\mc L_{\gamma_{1}}(\phi-\phi_{\gamma_1}).
$$
\end{itemize}
We can now evaluate the $L^1$ norm of the following difference
\begin{align}
\|\mc L_{\gamma_n}\mc L_{\gamma_{n-1}}&...\mc L_{\gamma_1}\phi-\phi_{\hg}\|_1\leq\nonumber\\\nonumber
&\leq \|\phi_{\gamma_n}-\phi_{\hg}\|_1+\|\mc L_{\gamma_n}...\mc L_{\gamma_{1}}(\phi-\phi_{\gamma_1})\|_1+\sum_{j=1}^{n}\|\mc L_{\gamma_n}...\mc L_{\gamma_{j}}(\phi_{\gamma_{j-1}}-\phi_{\gamma_j})\|_1 \\\nonumber
&\leq \|\phi_{\gamma_n}-\phi_{\hg}\|_1+\|\mc L_{\gamma_n}...\mc L_{\gamma_{1}}(\phi-\phi_{\gamma_1})\|_1+\sum_{j=1}^{n}\|\phi_{\gamma_{j-1}}-\phi_{\gamma_j}\|_1 \\\nonumber
\end{align}
where in the last inequality we used (P3).

\paragraph{Step 1} For every $\epsilon>0$, there exists  $\delta'>0$ and $\bar n=\bar n(\delta',\phi)$ such that for all $k\in\N$ and for all $\bo\gamma\in B_{\delta'}(\hg)^{kN}$:
\begin{equation}\label{multipN}
\|\mc L_{\gamma_{k\bar n}}\mc L_{\gamma_{k\bar n-1}}...\mc L_{\gamma_1}\phi-\phi_{\hg}\|_1\leq\epsilon
\end{equation}
We proceed by induction on $k$. Fix $\epsilon'<\epsilon/(2\bar n+2)$, and $\delta'$ such that $\|\phi_\gamma-\phi_{\gamma'}\|_1\leq\epsilon'$ for all $\gamma,\gamma'\in B_{\delta'}(\hg)$. In view of Lemma \ref{BounNormPertActMult} there exists a constant $K$, and $\tilde n(\phi,\epsilon')$ such that for all $n\geq \tilde n$ and all $\gamma_1\in B_{\delta'}(\hg)$ 
$$
\|\mc L_{\gamma_{n}}...\mc L_{\gamma_{1}}(\phi-\phi_{\gamma_1})\|_\alpha<M
$$
Choose $\bar n\geq \tilde n$ such that
\begin{equation}
\tilde q^{\bar n-\tilde n}M\leq \epsilon/2. \label{Nchoice}
\end{equation}
For $\bo\gamma\in B_{\delta'}(\hg)^{\bar n}$
\begin{align*}
\|\mc L_{\gamma_{\bar n}}\mc L_{\gamma_{\bar n-1}}...\mc L_{\gamma_1}(\phi-\phi_{\hg})\|_1&\leq  \epsilon'+\tilde q^{\bar n-\tilde n}\|\mc L_{\gamma_{\tilde N}}...\mc L_{\gamma_{1}}(\phi-\phi_{\gamma_1})\|_\alpha +\bar n\epsilon'\\
&\leq(\bar n+1)\epsilon'+\tilde q^{\bar n-\tilde n}M\\
&\leq \epsilon/2+\epsilon/2\leq \epsilon.
\end{align*}

Now call $\phi':=\mc L_{\gamma_{(k-1)\bar n}}...\mc L_{\gamma_1}\phi$. By the same argument:
\begin{align*}
\|\mc L_{\gamma_{k\bar n}}\mc L_{\gamma_{k\bar n-1}}...\mc L_{\gamma_{\bar n(k-1)+1}}\phi'-\phi_\hg\|_1&\leq (\bar n+1)\epsilon' +\tilde q^{\bar n-\tilde n}M\\
&\leq \epsilon.
\end{align*}

\paragraph{Step 2} Now choose $n>\bar n(\delta',\phi)$. There exists $k,r\in\N$ such that $n=k\bar n+r$. Call $\phi''=\mc L_{t_{k\bar n}}...\mc L_{t_{1}}\phi$. Using upper bound \eqref{multipN} we have
\begin{align*}
\|\mc L_{\gamma_n}...\mc L_{\gamma_1}\phi-\phi_\hg\|_1&=\|\mc L_{\gamma_n}...\mc L_{\gamma_{n-r+1}}\phi''-\phi_\hg\|_1 \\
&\leq \|\phi_{\gamma_n}-\phi_\hg\|_1+\sum_{j=1}^{r}\|\phi_{\gamma_{n-j}}-\phi_{\gamma_{n-j+1}}\|_1+|\phi''-\phi_{\gamma_{n-r+1}}\|_1\\
&\leq\epsilon'+ (r+1)\epsilon'+\epsilon\\
&\leq 2\epsilon.
\end{align*}
\end{proof}

We now consider the proof of Part (3). The main tool we use, along Theorem \ref{thmA},  is a law of large number for dependent random variables with sufficiently fast decaying correlations. 

\begin{theorem}[\cite{Walk2004}]\label{WalkerLLN}
Let $(X_n)_{n\in\N}$ be a sequence of square integrable random variables, such that there exists $r:\N_0\rightarrow\R$ with

 $$\left|
\mathbb{E}[(X_i - \mathbb{E} X_i ) (X_j - \mathbb{E} X_j ) ]
\right| \le r(|i-j|),\quad  i,j \in \mathbb{N}
$$
and
$$
\sum_{k=1}^{\infty} \frac{r(k)}{k} < +\infty. 
$$
then 
$$
\frac{1}{n} \sum_{k=1}^{n} (X_k - \mb E [X_k]) \rightarrow 0\quad \mbox{a.s.} 
$$
\end{theorem}
Given an observable $\psi\in V_\alpha$, $n\in\N$ and $\bo \gamma\in \Gamma'^\N$ (where $\Gamma'$ is as in Lemma \ref{BounNormPertActMult}), we define 
$$
\psi_0:=\psi \mbox{  and  }\psi_n:=U_{\gamma_1}... U_{\gamma_n}\psi
$$ 
and denote the sum of the observable over the orbit as
$$
S_n(\psi)(x)=\sum_{i=0}^{n-1}\psi_i(x).$$
 $(\psi_n)_{n\in \N_0}$ is a sequence of square integrable random variables on the measure space $(\R^N,m)$, and $ \psi_n-\mb E[\psi_n]=U_{\gamma_1}...U_{\gamma_n}(\tilde \psi_n)$ where 
$$
\tilde \psi_n:=  \psi-\int U_{\gamma_1}\circ...\circ U_{\gamma_n}\psi dm.
$$

We can estimate the covariance.
\begin{lemma}\label{DecCorLLN2} For any $\psi \in V_{\alpha}$, there exist a constant $C= C(\psi,\varepsilon_0,N)>0$ and $q < 1$ depending only on $F_{\hg}$ such that 
$$
\mb E\left[\left(\psi_i-\mb E[\psi_i]\right)\left(\psi_j-\mb E[\psi_j]\right) \right]\leq  C q^{j-i} .
$$
for every $\bo \gamma\in \Gamma'^\N$, $i,j\in\N$.
\end{lemma}
\begin{proof}
Without loss of generality, assume that $j>i$ and consider
\begin{align*}
R_{ij} & =  \mb E\left[\left(\psi_i-\mb E[\psi_i]\right)\left(\psi_j-\mb E[\psi_j]\right) \right]\\
&= \mb E\left[U_{\gamma_{1}}...U_{\gamma_i}\tilde\psi_i\cdot U_{\gamma_{1}}...U_{\gamma_j}\tilde\psi_j \right],
\end{align*}
then, using that properties of the Koopman and Transfer operator we obtain that

\begin{align*}
R_{ij}&=\mb E\left[\tilde\psi_i\cdot U_{\gamma_{i+1}}...U_{\gamma_j}\tilde \psi_j\cdot\mc L_{\gamma_i}...\mc L_{\gamma_1}1\right]\\
&=\mb E\left[ \tilde\psi_j \cdot  \mc L_{\gamma_{j}}...\mc L_{\gamma_{i+1}}\left(\tilde\psi_i\cdot\mc L_{\gamma_i}...\mc L_{\gamma_1}1  \right) \right].
\end{align*}
Since 
\begin{align*}
&\int  \tilde \psi_i\cdot\mc L_{\gamma_i}...\mc L_{\gamma_1}1 dm= \int U_{\gamma_1}...U_{\gamma_i}\left(\psi -\int U_{\gamma_1}... U_{\gamma_i}\psi dm\right)dm =0
\end{align*}
implies that $\tilde \psi_i\cdot\mc L_{\gamma_i}...\mc L_{\gamma_1}1 \in X_0$, and since $\mc L_{\gamma}$ restricted to $X_0$ are contractions with respect to the $\|\cdot\|_\alpha$ norm, we obtain the bound
\begin{align*}
 \| \mc L_{\gamma_{j}}...\mc L_{\gamma_{i+1}}\left(\tilde\psi_i\cdot\mc L_{\gamma_i}...\mc L_{\gamma_1}1  \right)\|_1&\leq \tilde q^{j-i}\|\tilde\psi_i\cdot\mc L_{\gamma_i}...\mc L_{\gamma_1}1\|_\alpha.
\end{align*}
By Proposition 3.4 in \cite{Saus}, $V_\alpha$ is an algebra and 
\begin{align}
\|\tilde\psi_i\cdot\mc L_{\gamma_i}...\mc L_{\gamma_1}1\|_\alpha&\leq C_\# \|\tilde\psi_i\|_\alpha\|\mc L_{\gamma_i}...\mc L_{\gamma_1}1\|_\alpha \label{algebraineq}\\
&\leq C_\#\left(|\psi|_\alpha+\|\tilde\psi_i\|_1\right)+\tilde \eta+\frac{c}{1-\tilde\eta} \label{unifboundosc}\\
&\leq C_\#	\left(1+\|\psi\|_\alpha\right) \nonumber
\end{align}
where $C_\#$ stands for an uninfluential constant uniform on $\psi$, $i$, and $\bo \gamma$. Inequality \eqref{algebraineq} is a consequence of the upper bound in Proposition 3.4 from \cite{Saus}; \eqref{unifboundosc} is derived from the first point in Lemma \ref{BounNormPertActMult}. We conclude noticing that, by H\"older inequality
\begin{align*}
\mb E\left[ \tilde\psi_j \cdot  \mc L_{\gamma_{j}}...\mc L_{\gamma_{i+1}}\left(\tilde\psi_i\cdot\mc L_{\gamma_i}...\mc L_{\gamma_1}1\right) \right]&\leq \|\tilde\psi_j\|_\infty\tilde q^{j-i}C_\#(1+\|\psi\|_\alpha)\\
&\leq  \|\psi\|_\infty\tilde q^{j-i}C_\#(1+\|\psi\|_\alpha)
\end{align*}

By Proposition 3.4 in \cite{Saus}, any $\psi\in V_\alpha$ is an essentially bounded function with

$$
\| \psi \|_{\infty } \le \frac{\max\{1 , \varepsilon^{\alpha}\}}{\gamma_N
\varepsilon_0^N} \| \psi  \|_{\alpha},
$$
where $\gamma_N$ is the volume of the $N-$dimensional unit ball.
Hence, taking 
$$
C:= C_\#  \frac{\max\{1 , \varepsilon^{\alpha}\}}{\gamma_N
\varepsilon_0^N} \| \psi  \|_{\alpha} (1+\|\psi\|_\alpha)
$$
we conclude the proof.
 \end{proof}

\begin{proof}[Proof of Part (3) of Theorem \ref{thmA}]
We know that
\begin{align*}
\mb E[\psi_k]&=\int\psi_kdm\\
&=\int U_{\gamma_1}...U_{\gamma_k}(\psi)dm\\
&=\int \psi \mc L_{\gamma_k}...\mc L_{\gamma_1}1dm
\end{align*}
and, using Theorem \ref{thmA},  for every $k>\bar n$ and sufficiently small perturbations 
$$
\|\mc L_{\gamma_k}...\mc L_{\gamma_1}1-\phi_0\|_1<\epsilon \quad\Rightarrow \quad\left|\mb E[\psi_k]-\int \psi \phi_0 dm\right|\leq \epsilon\mb E[|\psi|],
$$
and 
$$
\int \psi\phi_0dm-\epsilon\mb E[|\psi|]\leq\liminf_{n\rightarrow\infty}\frac{1}{n}\sum_k\mb E[\psi_k]\leq\limsup_{n\rightarrow\infty}\frac{1}{n}\sum_k \mb E[\psi_k]\leq \int \psi\phi_0dm+\epsilon\mb E[|\psi|].
$$
Thanks to Lemma \ref{DecCorLLN2} we can directly apply Theorem \ref{WalkerLLN} to the random variables $X_k= \psi_k$. This implies that for almost every $x\in M$
\begin{equation}\label{Eq:LLNResult}
\lim_{n\rightarrow\infty}\frac{1}{n}\sum_{k=0}^{n-1}\psi_k(x)-\mb E[\psi_k]=0.
\end{equation}
which implies that 
$$
\int \psi\phi_0dm-\epsilon\mb E[|\psi|]\leq\liminf_{n\rightarrow\infty}\frac{1}{n}\sum_k\psi_k(x)\leq\limsup_{n\rightarrow\infty}\frac{1}{n}\sum_k \psi_k(x)\leq \int \psi\phi_0dm+\epsilon\mb E[|\psi|].
$$
\end{proof}

\begin{proof}[Proof of Corollary \ref{Cor:WeakTop}]
To prove the corollary we use the Levy-Prokhorov metric for the weak topology. Given two probability measures $\mu,\nu$ on the Borel $\sigma-$algebra $(\Omega,\mc B(\Omega))$, the Levy-Prokhorov metric is defined as 
$$
d_{LP}(\mu,\nu):=\inf\left\{s|\quad\mu(A)\leq\nu(A_s)+s\quad\forall A\in\mc B\right\},
$$
where $A_s$ is the set of points at Euclidean distance strictly less than $s$ from $A$. Since the underlying topological space $\Omega$ is separable, it induces the weak topology on the space of Borel probability measures.
To prove the corollary is then enough to show that fixed $\epsilon>0$ there is a $\delta>0$ such that for every sequence $\bo\gamma\in B_\delta(\hg)^{\N}$, and for almost every $x\in \Omega$ there is $\bar n$ such that  
$$
\mu_{n,x}:=\frac{1}{n}\sum_{i=0}^{n-1}\left(F_{\bo \gamma}^i\right)_*\delta_{x}\in B_\epsilon(\mu_\hg),\quad \forall n>\bar n,
$$ 
where $B_\epsilon(\mu_\hg)$ is the $\epsilon$-ball around $\mu_\hg$ w.r.t. the metric $d_{LP}$.

$$
\mu_{n,x}(A)=\frac{1}{n}\sum_{i=0}^{n-1}\chi_{A}\circ F^i_{\bo\gamma}(x) 
$$
and, from equation \eqref{Eq:LLNResult} in the proof of  Part (3) of Theorem A with $\psi=\chi_A$,  
$$
\mu_{n,x}(A)-\frac{1}{n}\sum_i\mb E\left[\chi_A\circ F^i_{\bo\gamma}(x) \right]\rightarrow 0 
$$
for almost every $x\in\Omega$. Using the properties of the transfer operator
\begin{equation}\label{Eq:conv}
\mu_{n,x}(A)-\frac{1}{n}\sum_{i=0}^{n-1}\int_A\mc L_{\bo\gamma}^i(1)dm\rightarrow 0 \quad\mbox{a.e.}
\end{equation}
$L^1$ convergence of densities implies weak convergence of the corresponding probability measures, thus
\begin{equation}\label{Eq:resultofTheo}
\|\mc L^i_{\bo\gamma}(1)-\phi_\hg\|_1\leq \epsilon'\Rightarrow d_{LP}\left((F^i_{\bo\gamma})_*m,\mu_\hg\right)\leq \Delta(\epsilon')
\end{equation}
with $\Delta(\epsilon')\rightarrow 0$ for $\epsilon'\rightarrow 0$ which implies
\begin{equation}
\int_A\mc L_{\bo\gamma}^i(1)dm=(F^i_{\bo\gamma})_*m(A)\leq \mu_{\hg}(A_{\Delta(\epsilon')})+\Delta(\epsilon')\quad \forall A\in\mc B
\end{equation}

Now the main problem is that convergence \eqref{Eq:conv} is not uniform in $A$. To overcome this issue, we use the standard technique to approximate any measurable set $A$ as the union of sets taken from a sufficiently fine (in terms of the Euclidean metric), but finite collection, and use uniformity of \eqref{Eq:LLNResult} for this finite collection to deduce uniformity for any $A$. Fix $\delta'>0$. From the topological properties of $\Omega$, one can find a finite collection of balls $\{B_i\}_{i=1}^J$ such that: $\diam(B_i)<\delta'$, $m(\cup_{i=1}^J B_i)>1-\delta'$. The set 
$$
\mc D:=\{\cup_{i\in\mc J}B_i|\quad \mc J\subset \{1,...,J\}\}, 
$$ 
given by all possible unions of sets from the collection, is itself finite. This means that we can find a subset $\Omega_{\delta'}\subset \Omega$ of full measure, such that for every $x\in \Omega_{\delta'}$ there is $\bar n=\bar n(x)$ such that 
\begin{equation}\label{Eq:UniformBoundMeas}
\left|\mu_{n,x}(B)-\frac{1}{n}\sum_{i=0}^{n-1}\int_BP_{\bo\gamma}^i(1)dm\right|<\delta', \quad \forall B\in\mc D\mbox{ and }\forall n>\bar n'.
\end{equation}
From Theorem A, we can choose  $\delta>0$ so small that \eqref{Eq:resultofTheo} holds for every $i>N_1$ and $\Delta(\epsilon')<\delta'$. Thus there is a $N_2>N_1$ such that: 
\begin{equation}\label{Eq:upperboundLevProktype}
\left|\frac{1}{n}\sum_{i=0}^{n-1}\int_BP_{\bo\gamma}^i(1)-\mu_{\hg}(B_{\Delta(\epsilon')})\right|<2\Delta(\epsilon')<2\delta'\quad \forall n>N_2.
\end{equation}
Putting \eqref{Eq:upperboundLevProktype} and \eqref{Eq:UniformBoundMeas} together 
$$
\mu_{n,x}(B)\leq \mu_{\hg}(B_{\delta'})+3\delta'.
$$
Now consider $\mc J_A:=\{i| \quad B_i\cap A\neq 0\}$. It holds that: $A\subset \left(\cup_{i\in\mc J_A}B_i\right) \cup\left(\cup_{i=1}^JB_i\right)^c$ (trivially); and $ \left(\cup_{i\in\mc J_A}(B_i)_{\delta'}\right)\subset A_{2\delta'}$ (from the condition on the diameters of the sets of the partition).
\begin{align*}
\mu_{x,n}(A)&\leq \mu_{x,n}(\cup_{i\in\mc J_A}B_i)+\delta'\\
&\leq \mu_{\hg}(\cup_{i\in\mc J_A}(B_i)_{\delta'})+4\delta'\\
&\leq \mu_{\hg}(A_{2\delta'})+4\delta'
\end{align*}
which for the right choice of $\delta'$ gives the desired result.
\end{proof}

\begin{appendices}
\section{$C^{1+\nu}$ Expanding Maps }
In this appendix we consider a collection $ \{F_\gamma\}_{\gamma\in \Gamma}$ of $C^{1+\nu}$ maps from $M$ into itself (where $\Gamma$ is some
metric space) which are continuous at $\hg$, i.e. so that  the map $\gamma\mapsto F_\gamma$ from $\Gamma$ to $C^{1+\nu}(M,M)$ is
{\em continuous at $\hg$}. 
We will assume that $F=F_{\hg}$ is expanding and for 
each sequence  $\bo \gamma=(\gamma_1,\gamma_2,\dots)\in (B_{\delta}(\hg))^{\N}$ we analyse the ergodic properties of 
compositions of the form
$$F^n_{\bo \gamma}=F_{\gamma_n} \dots  F_{\gamma_1}.$$ 
We are particularly interested to study what happens when we let $\delta$ tend to zero.

\subsection{Setting and Result}
Let $M$ be connected and compact Riemannian manifold, with Riemannian distance $d$ and Riemannian volume $m$ (we normalise so that $m(M)=1$). 
\begin{definition}\label{ExpSmoothMap}
 We say that $F:M\rightarrow M$ is a $C^{1+\nu}$ {\em expanding map} if $F$ is  differentiable, $\log |\det D_xF|$ is a locally $\nu$-H\"older function and there exists a constant $\sigma\in(0,1)$ such that
$$
\|D_xF(v)\|\geq\sigma^{-1}\|v\|\quad\quad \forall x\in M,\quad \forall v\in T_xM.
$$
\end{definition}

\medskip It is well-known that 
$F$ has an {\em absolutely continuous invariant measure} $\mu_0$ 
with a density $\phi_0$ which is strictly positive and $\nu$-H\"older.

\bigskip

The next theorem shows that for expanding maps all ergodic properties are robust:

\begin{theo}
 \label{thmB} 
Let $F$ be a $C^{1+\nu}$-expanding map and $\{F_{\gamma}\}_{\gamma\in \Gamma}$ a collection of $C^{1+\nu}$ perturbed versions continuous at $\hg$. Then for every $\epsilon>0$, there exists $\delta>0$ so that: 

\begin{enumerate}
\item[(1)] \cite{ViaSdds}
if $\gamma\in B_{\delta}(\hg)$ then $F_\gamma$ has an invariant density $\phi_\gamma$, and
$$
\sup_{x\in M}|\phi_\gamma(x)-\phi_{\hg}|<\epsilon;
$$

 \item[(2)] for every $\bo\gamma\in B_{\delta}(\hg)^{\N}$ and probability measure $\mu=\phi m$ with a strictly positive $\nu$-H\"older density $\phi$, there exists $\bar n$ so that for $n>\bar n$, 
the Radon-Nykodim derivative 
$
\frac{d}{dm}(F^n_{\bo \gamma})_*\mu
$
is strictly positive and $\nu$-H\"older and 

\begin{equation}
\sup_{x\in M}\left| \left(\frac{d}{dm} (F^n_{\bo \gamma})_*\mu\right)(x)-\phi_{\hg}(x)\right|<\epsilon;
\label{eq:ThmA}\end{equation}

\item[(3)] there exists a set $X_{\bo \gamma}\subset M$ of full measure such that for each $\nu$-H\"older observable $\psi$ and for each $x\in X_{\bo \gamma}$ 
$$
\int \psi\phi_0 dm -\epsilon\int |\psi| dm \leq \liminf_{n\rightarrow\infty}\frac{1}{n}S_n(\psi)(x)\leq \limsup_{n\rightarrow\infty}\frac{1}{n}S_n(\psi)(x)\leq \int \psi \phi_0 dm+\epsilon\int |\psi| dm,
$$
where $S_n(\psi)(x)=\psi(x)+\sum_{i=1}^{n-1}\psi (F_{\gamma_i}...  F_{\gamma_1}(x))$. 
\end{enumerate}
\end{theo}

\begin{remark}
Recalling that the unique a.c.i.p. density for a $C^{1+\nu}$ expanding map on compact manifold is uniformly bounded and bounded away from zero, the convergence in the $\sup$-norm stated above implies that after a finite number of iterates, any regular density will remain bounded and bounded away from zero under any sufficiently small nonautonomous composition of perturbations of such a map.
\end{remark}

\begin{remark}
Keller \cite{KellPro} obtained a result $L^1$-analogue of  inequality (\ref{eq:ThmA})  for piecewise expanding interval maps. 
In our proof of Theorem B we use the well-known cone approach (explained below). In our setting, since we assume the maps are  $C^{1+
\nu}$, the cone approach is more direct and gives the stronger uniform (rather than $L^1$) estimates.
\end{remark}

\subsection{Strategy of the proof of Theorem B}

The main ingredient in the  proof of Theorem B is the  standard cone approach (\cite{Liv, Liv4,ViaSdds}). 
For sufficiently small perturbations of a given expanding map, the associated transfer operators leave a cone of strictly positive continuous functions with H\"older logarithm invariant. Endowing the cone with the Hilbert metric, one is able to prove that it has finite diameter and thus the restrictions of the transfer operators to the cone are contraction with respect to this metric and their contracting rates are uniformly bounded away from 1. This immediately implies memory loss for initial distributions of states belonging to the cone. The main claim is implied by a combination of this last result with continuity of the map $t\mapsto \mc L_t\phi$  for every fixed function $\phi$ inside the cone. We preliminarily define the transfer operator, cones of functions and the Hilbert metric in Section \ref{Sec:prel} to state results of existence of a.c.i.p. measures in Section \ref{Smooacipm} and stochastic stability in Section \ref{SecSmoStatStab}. We conclude with the proof of the result in the final section.

\label{ThmAproof}
\subsection{Preliminaries}\label{Sec:prel}

In this section, we review some of the main concepts and techniques to study invariant measure and statistical properties of regular expanding maps that will lead to the proof of Theorem \ref{thmB}  given in Section \ref{ProofNonatres}.

\subsubsection{Transfer Operator}\label{Sec:TranOper}
Given a measurable space $(M,\mc B, m)$ with a nonsingular transformation $F$ one can define two operators on the spaces $L^{\infty}(M)$ and $L^1(M)$ (the specification of the measure $m$ is omitted whenever there is no risk of confusion), respectively $U:L^\infty(M)\rightarrow L^\infty(M)$
$$
U(\phi)=\phi\circ F
$$
and its adjoint $\mc L:L^1(M)\rightarrow L^1(M)$, that satisfies for all $\phi \in L^\infty(M)$ and $\psi\in L^1(M)$
\begin{equation}\label{CharacPropTranOp}
\int_M U\phi\cdot\psi dm=\int_M\phi\cdot\mc L\psi dm.
\end{equation}
$\mc L$ is the transfer (or Perron-Frobenius) operator, and has the following properties:
\begin{align}
& \text{positivity: }  &\psi\geq0& \Rightarrow \mc L_{F}\psi\geq 0 \tag{P1}\\
& \text{preserves integrals: } &\int \mc L_F \psi \diff m&=\int \psi \diff m \tag{P2} \\
& \text{contraction property: } &|\mc L_F\psi|_1&\leq |\psi|_1 \tag{P3} \label{P3}\\
& \text{composition property: }& \mc L_{F\circ G}&=\mc L_G\circ \mc L_F\tag{P4}
\end{align}
for any $\psi\in L^1(M)$.

The transfer operator has proven to be an invaluable tool to deduce statistical properties of dynamical systems (such as existence of invariant measures, decay of correlations, and central limit theorems  for Birkhoff sums \cite{Liv, Bal, BoyGor,LuzMel,ViaSdds}) and their perturbations (\cite{Kel,KellPro,BalYou,ViaSdds}). For example, the probability density fixed by the transfer operator are the densities of the invariant absolutely continuous probability measures, and likewise, for any $\mu=\phi m$ a.c.i.p. measure, $\phi$ is a fixed point for the transfer operator. The transfer operator also prescribes the evolution of the densities.

For maps satisfying Definition \ref{ExpSmoothMap}, each point of of $M$ has the same finite number of preimages under $F$ ($M$ is compact and connected), and the transfer operator for these maps can be written as
$$
\mc L \phi(x)=\sum_{i=1}^k\phi(y_i)\cdot |\det D_{y_i}F|^{-1},\quad\quad F^{-1}(x)=\{y_1,...,y_k\}.
$$

\subsubsection{Projective Metric}

One of the main tools to investigate properties of the transfer operator is the Hilbert Projective metric \cite{Liv, Liv4,ViaSdds}.

Given a linear vector space $E$, a cone of $E$ is a subset $C\subset E\backslash\{0\}$ such that if $v\in C$ then $sv\in C$ for all $s>0$. We consider  convex cones, namely those  that satisfy
$$
s_1 v_1+s_2 v_2\in C\quad\quad \forall s_1,s_2>0\mbox{ and }\forall v_1,v_2\in C.
$$    
There is a canonical way to define a pseudo metric on every cone. Let
$$
\alpha(v_1,v_2):=\sup\{s>0:\mbox{ }v_2-sv_1\in C\},
$$
$$
\beta(v_1,v_2):=\inf\{s>0:\mbox{ } sv_1-v_2\in C\},
$$
and 
$$
\theta(v_1,v_2):=\log{\frac{\beta(v_1,v_2)}{\alpha(v_1,v_2)}}
$$
with the logarithm extended to a function of $[0,+\infty]$. $\theta$ has the following properties (see \cite{ViaSdds} for a proof):
\begin{itemize}
\item[(i)] $\theta(v_1,v_2)=\theta(v_2,v_1)\mbox{, }\forall v_1,v_2\in C$,
\item[(ii)] $\theta(v_1,v_3)\leq\theta(v_1,v_2)+\theta(v_2,v_3)\mbox{, }\forall v_1,v_2,v_3\in C$,
\item[(iii)] $\theta(v_1,v_2)=0$ if and only if there is $s>0$ s.t. $v_1=sv_2$.
\end{itemize}

 Point (iii) implies that $\theta$ distinguishes directions only, and for this reason it is called projective metric. 
\begin{remark}
Looking closely at $\alpha$ and $\beta$, one can notice that it is necessary to evaluate one of the two only. In particular
\begin{equation}\label{alpbetrel}
\alpha(v_1,v_2)=\frac{1}{\beta(v_2,v_1)}\quad \forall v_1,v_2\in C
\end{equation}
which also implies property (i) of $\theta$.
\end{remark}

\begin{remark}
The projective metric $\theta:=\theta_C$ depends on the cone on which is defined. Vectors belonging to the intersection of two different cones of the same vector space might have different projective distances wether considered as vectors of the first cone or of the second.
\end{remark}

Given cones $C_1\subset C_2$, with projective metrics $\theta_1$ and $\theta_2$, one has
\begin{equation}\label{projmetrconsubdec}
\theta_2(v_1,v_2)\leq\theta_1(v_1,v_2),\quad \forall v_1,v_2\in C_1,
\end{equation}
which tells us that the projective distance between two vectors decreases after enlarging the cone.
We now present two examples that will be useful in what follows.  

\begin{example} The cone of strictly positive continuous functions is
$$
C_+:=\{\phi\in C^{0}(M,\R)\mbox{ s.t. } \phi>0\}.
$$
We compute the projective metric.
\begin{align*}
\alpha_+(\phi_1,\phi_2) &=\sup\{t>0:\mbox{ } \phi_2(x)-t\phi_1(x)>0,\quad\forall x\in M\}\\
&=\sup\{t>0:\mbox{ } t<\phi_2(x)/\phi_1(x),\quad \forall x\in M\}
\end{align*}
which gives $\alpha_+(\phi_1,\phi_2)=\inf_{x\in M}\phi_2(x)/\phi_1(x)$.
From \eqref{alpbetrel}, $\beta_+(\phi_1,\phi_2)=\sup_{x\in M}\phi_2(x)/\phi_1(x)$, and
$$
\theta_+(\phi_1,\phi_2)=\log\sup\left\{\frac{\phi_2(x)\phi_1(y)}{\phi_1(x)\phi_2(y)}\mbox{ s.t. }x,y\in M\right\}.
$$
\end{example}

\begin{example}\label{ExHoldCon} For every $a>0$ and $\nu\in(0,\nu)$, consider the following cone
\begin{equation}\label{HoldLogCon}
C(a,\nu):=\left\{\phi\in C_+\mbox{ s.t. }d(x_1,x_2)\leq\rho_0\Rightarrow \phi(x_1)\leq\exp(ad(x_1,x_2)^\nu)\phi(x_2)\right\}
\end{equation}
which is the cone of strictly positive continuous functions with $\log\phi$ locally $\nu-$H\"older. The computation of the projective metric $\theta_{a,\nu}$ on this cone can be obtained in a similar way as before (see \cite{ViaSdds} for details). In particular we have
\begin{equation*}
\alpha(\phi_1,\phi_2)=\inf\left\{\frac{\phi_2(x)}{\phi_1(x)},\frac{\exp(ad(x,y)^\nu) \phi_2(x)-\phi_2(y)}{\exp(ad(x,y)^\nu) \phi_1(x)- \phi_1(y)}\mbox{ s.t. }x,y\in M\mbox{ and }0<d(x,y)<\rho_0\right\}.
\end{equation*}
\end{example}

One usually consider cones instead of the whole linear space because the restriction of linear operators to invariant cones exhibits nice properties. For example, letting $E_1,E_2$ be two vector spaces, $L:E_1\rightarrow E_2$ a linear map, and $C_1,C_2$ two cones in $E_1$ and $E_2$ respectively such that $L(C_1)\subset C_2$, it is easy to verify that 
$$
\theta_2(L(v_1),L(v_2))\leq \theta_1(v_1,v_2) \quad \forall v_1,v_2\in C_1.
$$
Furthermore, if the image of $C_1$ has finite dimeter, then the restriction of the linear map to the cone is a contraction.
\begin{proposition}[\cite{ViaSdds}]\label{FinDimImpCont}

Let $L:E_1\rightarrow E_2$ be a linear map, and $C_1\subset E_1$ a cone. If 
$$
D:=\sup\{\theta_2(L(v_1),L(v_2)) \mbox{ s.t. }v_1,v_2\in C_1 \}
$$ 
is finite, then
$$
\theta_2(L(v_1),L(v_1))\leq q\theta_1(v_1,v_2)
$$
with $q=(1-e^{-D})$.
\end{proposition}

\subsubsection{Absolutely Continuous Invariant Probability Measure}\label{Smooacipm}

The above machinery has been used to prove existence of an invariant absolutely continuous probability measure for the class of maps introduced in Definition \ref{ExpSmoothMap}. 
 \begin{proposition}[\cite{ViaSdds}]
Let $F$ be a map as in Definition \ref{ExpSmoothMap} , and let $\mc L$ be the associated transfer operator. Then, for all sufficiently large $a>0$, for $0<\nu<\nu$, and for all $\lambda\in(\sigma,1)$  
$$
\mc L(C(a,\nu))\subset C(\lambda a,\nu).
$$

Moreover, the diameter 
\begin{equation}\label{diametercone}
D_{\lambda a,\nu}:=\sup\{\theta_{a,\nu}(\phi_1,\phi_2) \mbox{ s.t. } \phi_1,\phi_2\in C(\lambda a,\nu)\}
 \end{equation}
 is finite for every $a>0$, $\nu>0$, $0<\lambda<1$.
\end{proposition}

This proposition along with Proposition \ref{FinDimImpCont} imply that the action of $\mc L $ contracts directions inside the cone $C(a,\nu)$. To get a step closer to obtaining a fixed point, we can restrict to the subset of normalised densities in $C(a,\nu)$,
$$
\tilde C(a,\nu):=\{\phi\in C(a,\nu)\mbox{ s.t. }\int\phi(x)dm(x)=1\},
$$ 
and show that $\mc L$ is a contraction of this space with respect to the restriction of the projective metric.
\begin{proposition}
The following holds:
\begin{itemize}
\item[(i)] The restriction of $\theta_{a,\nu}$ to $\tilde C(a,\nu)$ is a metric,
\item[(ii)] $\mc L(\tilde C(a,\nu))\subset \tilde{C}(\lambda a,\nu)$.
\end{itemize}
\end{proposition}  
\begin{proof}
(i): $ C(a,\nu)$ has finite dimeter, thus $\theta_{a,\nu}$ takes only finite values. $\theta_{a,\nu}(\phi_1,\phi_2)=0$ for $\phi_1,\phi_2\in \tilde C(a,\nu)$ implies $\phi_1=c\phi_2$, but $c$ must be equal to one so $\phi_1=\phi_2$.
(ii): implied by properties (P1) and (P2) of the transfer operator.
\end{proof}
This immediately gives the following corollary. 
\begin{corollary}
$\mc L$ is a contraction on the metric space $(\tilde C(a,\nu),\theta_{a,\nu})$.
\end{corollary}

Without getting into much details (which can be found in \cite{ViaSdds}), to produce a fixed point, it would  be sufficient to show that every normalised Cauchy sequence in $C(a,\nu)$ is convergent. The setback is that $(\tilde C(a,\nu),\theta_{a,\nu})$ is not complete. However, as it has already been pointed out, $C(a,\nu) $ is a subset of $C_+$. Normalised Cauchy sequences converge in this cone with respect to $\theta_+$, and the fact that $\theta_+(\phi_1,\phi_2)\leq\theta_{a,\nu}(\phi_1,\phi_2)$  for all $\phi_1,\phi_2\in C(a,\nu)$ (see equation \eqref{projmetrconsubdec}), implies the existence of a fixed point $\phi_0\in C_+$ which is an invariant density for the dynamical system $(M,\mc B, m, F)$. One is also able to prove that $\phi_0$ is in fact a function in $C(\lambda a,\nu)$ and that there exists constants $R>0$ and $\sigma_1\in(0,1)$ such that
\begin{equation}\label{ExpConvInvMeas}
\sup_{x\in M}|\mc L^n\phi(x)-\phi_0(x)|<R\sigma_1^n
\end{equation}
for all densities $\phi\in \tilde C(a,\nu)$.
This means that any distribution of mass $\mu=\phi m$ on $M$, for $\phi\in \tilde C(a,\nu)$,  will evolve exponentially fast towards the invariant distribution $\mu_0=\phi_0 m$ under the iteration of the map.
\begin{remark}\label{invariancconesalltrans}
Notice that taking a collection of $C^{1+\nu}$ perturbations $\{F_\gamma\}_{\Gamma}$, if $\delta>0$ is sufficiently small, then every $F{_\gamma}$ for $\gamma$ in the open ball $B_\delta(\hg) $ is a $C^{1+\nu}$ expanding map with uniform lower bound $\bar \sigma^{-1}$ on the rate of expansion, and on the H\"older constant. This implies that all the associated transfer operators map $C(a,\nu)$ into $C(\bar \lambda a,\nu)$, for some sufficiently large $a$ and $\bar\lambda\in(0,1)$. This implies that for all $\gamma \in B_{\delta}(\hg) $, $\mc L_\gamma$ has a fixed point $\phi_\gamma\in C(\bar \lambda a,\nu)$,  and therefore there is an invariant absolutely continuous probability measure for the perturbed map. The analogous of \eqref{ExpConvInvMeas} holds:
\begin{equation}\label{ExpConvInvMeasPert}
\sup_{x\in M}|\mc L_\gamma^n\phi(x)-\phi_\gamma(x)|<R\sigma_2^n
\end{equation}
for some $R>0$, $\sigma_2\in(0,1)$ and $ \forall \gamma\in B_\delta(\hg)$.
\end{remark}

\subsection{Stochastic Stability}

As for the piecewise case (Section \ref{SecSmoStatStab}), also for families of $C^{1+\nu}$ maps one can consider independent compositions of maps sampled according to some measure $\nu$ on $\Gamma$.  In this case, it is known (\cite{ViaSdds,BalYou}) that if the support of $\nu$ is sufficiently close to $\tau$, i.e. $\supp \mu\subset B_\delta(\tau)$ for sufficiently small $\delta>0$, then the averaged transfer operator has an invariant density $\phi_\mu$, and this invariant density converges uniformly to the invariant density $\phi_\tau$ whenever $\delta\rightarrow 0$.  $\phi_\mu$ is called a stationary density and it describes the asymptotic distribution for the random orbits for almost every initial condition (w.r.t. $\phi_\mu m$), and for almost every sequence $\{t_i\}_{i\in \N}$ (w.r.t. $\mu^{\otimes \N}$). The stochastic stability result asserts that stationary densities are uniformly close to the unperturbed invariant density, whenever the support of their measure is sufficiently close to $\tau$.
\begin{proposition}[\cite{ViaSdds}]\label{StocStabSmooth}
Given a probability measure $\mu$ on $T$, for every $\epsilon>0$ there is $\delta>0$, such that if $\supp\mu\subset B_\delta(\tau)$ then
$$
\sup_{x\in M}|\phi_\mu(x)-\phi_\tau(x)|<\epsilon.
$$
\end{proposition}
\begin{remark}
Notice that a particular case of the above setting is when $\mu=\delta_t$ is the singular probability measure concentrated at the point $t\in T$. In this case, $\hat {\mc L}_\mu$ equals $\mc L_t$, and the above results imply that for $t$ sufficiently close to $\tau$, $F^{(t)}$ has an absolutely continuous invariant probability measure with density $\phi_t$, and $\phi_t\rightarrow \phi_\tau$ uniformly for $t\rightarrow \tau$.
\end{remark}

The bound \eqref{ExpConvInvMeasPert} implies that the evolution of densities under the iterated action of a perturbed map $F_\gamma$ converges exponentially fast to the invariant density $\phi_\gamma$. Proposition \ref{StocStabSmooth}  with the above remark imply that for small perturbations, under iteration of some map $F_\gamma$, densities in $\tilde C(a,\nu)$, evolve close to $\phi_\hg$.

\subsection{Proof of Theorem B}\label{ProofNonatres}\label{ThmAproof}

Part (1) of Theorem B is given by Proposition \ref{StocStabSmooth}. Now we prove part (2) that, in contrast with the stationary case, tells us what happens to densities when we apply perturbed versions of the map $F$ without requiring any kind of independence of the perturbations.

\begin{proof}[Proof of part (2) of Theorem \ref{thmB}]

Thanks to property \eqref{CharacPropTranOp} we can restate the theorem in terms of the action of the transfer operator on the densities. Thus we need to prove that for $\{F_{\gamma}\}_{\gamma\in \Gamma}$ as in the hypotheses, denoted with 
$$
\mc L^n_{\bo \gamma}:=\mc L_{\gamma_n}...\mc L_{\gamma_1}
$$
the transfer operator of the composition $F^n_{\bo \gamma}:=F_{\gamma_n}\circ...\circ F_{\gamma_1}$, for every $\epsilon>0$ there exists $\delta(\epsilon)>0$ and $\bar n(\epsilon)\in\N$ such that for every $\phi\in\tilde C(a,\nu)$, every $n>\bar n$ and $\bo \gamma\in B_\delta(\hg)^n$

$$
\mc L^n_{\bo \gamma}\phi\in \tilde C(a,\nu)\quad\mbox{ and }\quad\sup_{x\in M}|\mc L^n_{\bo \gamma}\phi(x)-\phi_0(x)|<\epsilon.
$$
$\mc L^n_{\bo \gamma}\phi\in \tilde C(a,\nu)$ follows from Remark \ref{invariancconesalltrans}.  

Then, we recall that, for every $\epsilon'>0$, $\phi\in C(a,\nu)$, and every $n\in\N$ there exists $\delta(\epsilon',\phi,n)>0$ such that if $\bo{\gamma}\in B_\delta(\tau)^{n}$ 
$$
\theta_+(\mc L_{\bo \gamma}^n\phi,\mc L^n\phi)\leq \epsilon'
$$
This is proven in Proposition 2.14 of \cite{ViaSdds} for the case $n=1$, and can be generalised to any finite $n\in\N$.

If $\gamma\in B_\delta(\hg)$ with $\delta>0$ sufficiently small, for some $a>0$ and $0<\nu<1$, $\mc L_\gamma$ is a contraction of $(\tilde C(a,\nu),\theta_{a,\nu})$ with contracting constant $q\in(0,1)$ independent of $\gamma$. Fix $\epsilon'>0$, and choose $\bar n\in \N$ so that $q^{\bar n} D_{\lambda a,\nu}<\epsilon'/2$, with $D_{\lambda a,\nu}$ diameter of $C(a,\nu)$ (see Eq. \eqref{diametercone}), and $\delta$ less than $\delta(\epsilon'/2,\phi_0,\bar n)$ above, so that $\theta_+(\mc L_{\bo \gamma}^{\bar n}\phi_0,\phi_0)<\epsilon'/2$ for all $\bo \gamma\in B_{\delta}(\hg)^{\bar n}$. This implies that for any $n>\bar n$ and any $\bo \gamma\in B_\delta(\hg)^{\bar n}$
\begin{align}
\theta_+(\mc L^n_{\bo \gamma}\phi, \phi_0)&\leq \theta_+(\mc L_{\bo \gamma}^n\phi,\mc L_{\gamma_n}...\mc L_{\gamma_{n-\bar n}}\phi_0)+\theta_+(\mc L_{\gamma_n}...\mc L_{\gamma_{n-\bar n}}\phi_0, \phi_0)\label{ResSmoINeq1}\\
&\leq \theta_{a,\nu}(\mc L_{\bo \gamma}^n\phi,\mc L_{\gamma_n}...\mc L_{\gamma_{n-\bar n}}\phi_0)+\theta_+(\mc L_{\gamma_n}...\mc L_{\gamma_{n-\bar n}}\phi_0, \phi_0)\label{ResSmoINeq2}\\
&\leq \theta_{a,\nu}(\mc L_{\gamma_n}...\mc L_{\gamma_{n-\bar n}}\mc L_{\gamma_{n-\bar n-1}}...\mc L_{\gamma_1}\phi,\mc L_{(\gamma_n,...,\gamma_{n-\bar n})}\phi_0)+ \epsilon'/2\\
&\leq q^{\bar n}D_{\lambda a,\nu}+ \epsilon'/2\label{ResSmoINeq3}\\
&\leq \epsilon'\nonumber
\end{align}

Where \eqref{ResSmoINeq1} is just triangular inequality, \eqref{ResSmoINeq2} follows from inequality \eqref{projmetrconsubdec} that upper bound the metric $\theta_+$ with the metric $\theta_{a,\nu}$, and \eqref{ResSmoINeq3} follows from uniformity of the contraction rate, $q\in(0,1)$, for $\{\mc L_\gamma\}_{\gamma\in B_{\delta}(\hg)}$.
\end{proof}

To prove Part (3) we use again the strong law of large numbers in Theorem \ref{WalkerLLN}. Given an observable $\psi\in C(a,\nu)$, $n\in\N$ and $\bo \gamma\in \Gamma^\N$, we define 
$$
\psi_0:=\psi\mbox{, }\quad \psi_n:=U_{\gamma_1}...U_{\gamma_n}\psi\quad  \mbox{ and }\quad  S_n(\psi)(x): =\sum_{i=0}^{n-1}\psi_i(x).
$$ 
$(\psi_n)_{n\in \N_0}$ is a sequence of square integrable random variables on the measure space $(M,m)$, and $ \psi_n-\mb E[\psi_n]=U_{\gamma_1}...U{\gamma_n}(\tilde \psi_n)$ where 
\begin{align*}
\tilde \psi_n(x)&:=  \psi-\int U_{\gamma_1}...U_{\gamma_n}\psi dm\\
&= \psi-\int\psi \mc L^n_{\bo\gamma}( 1)dm.
\end{align*}
We  estimate the covariance in the following lemma.
\begin{lemma}\label{DecCorLLN}
There exists constants $R>0$ and $r\in(0,1)$ such that for every $i\neq j$
$$
\mb E\left[\left(\psi_i-\mb E[\psi_i]\right)\left(\psi_j-\mb E[\psi_j]\right) \right]\leq Rr^{|j-i+1|}.
$$
\end{lemma}
\begin{proof}
Suppose that $j>i$ without loss of generality.
\begin{align*}
&\phantom{==}\mb E\left[\left(\psi_i-\mb E[\psi_i]\right)\left(\psi_j-\mb E[\psi_j]\right) \right]=\\
&= \mb E\left[U_{\gamma_{1}}...U_{\gamma_i}\tilde\psi_i\cdot U_{\gamma_{1}}...U_{\gamma_j}\tilde\psi_j \right]\\
&=\mb E\left[\tilde\psi_i\cdot U_{\gamma_{i+1}}...U_{\gamma_j}\tilde \psi_j\cdot\mc L_{\bo\gamma}^i1\right]\\
&=\mb E\left[ \tilde\psi_j \cdot  \mc L_{\gamma_{j}}...\mc L_{\gamma_{i+1}}\left(\tilde\psi_i\cdot\mc L_{\bo\gamma}^i1  \right) \right]\\
&=\int \tilde\psi_j (x)\cdot\mc L_{\gamma _j}...\mc L_{\gamma_{i+1}}\left(\psi\cdot\mc L_{\bo\gamma}^i1-\mc L_{\bo\gamma}^i1\int\psi \mc L_{\bo\gamma}^i1dm\right)(x)dm(x).
\end{align*}

Now, $\phi_1:=\mc L_{\bo\gamma}^i1(x)\cdot\psi(x)$ and $\phi_2:=(\int\psi \mc L_{\bo\gamma}^i1dm)\cdot\mc L_{\bo\gamma}^i1(x)$ are  positive densities with the same expectation, and both belong to $C(2a,\nu)$. In fact, $\phi_2\in C(a,\nu)\subset C(2a,\nu)$, and $\phi_1$ is the product of two densities in $C(a,\nu)$ which is a density of $C(2a,\nu)$. So, letting $r\in(0,1)$ be the uniform contraction rate of the operators $\{\mc L_\gamma\}_{\gamma\in\Gamma}$ on $C(2a,\nu)$
\begin{align*}
\theta_+(\mc L_{\gamma_j}...\mc L_{\gamma_{i+1}}(\phi_1),\mc L_{\gamma_j}...\mc L_{\gamma_{i+1}}(\phi_2))&\leq \theta_{2a,\nu}(\mc L_{\gamma_j}...\mc L_{\gamma_{i+1}}(\phi_1),\mc L_{\gamma_j}...\mc L_{\gamma_{i+1}}(\phi_2))\\
&\leq r^{j-i-1}\theta_{2a,\nu}(\phi_1,\phi_2)
\end{align*}
and
\begin{align*}
&\sup_{x}\left|\mc L_{\gamma_j}...\mc L_{\gamma_{i+1}}\left(\mc L_{\bo\gamma}^i1(x)\psi(x)\right)-\mc L_{\gamma_j}...\mc L_{\gamma_{i+1}}\left(\mc L_{\bo\gamma}^i1(x)\int\psi \mc L_{\bo\gamma}^i1dm\right)\right|\leq R_1 [\exp(R_2r^{j-i+1})-1] 
\end{align*}
with 
\begin{align*}
R_1&:=\sup\{\phi(x)|\quad x\in M,\phi\in C(2a,\nu) \}\\
R_2&:=\sup\{\theta_{2a,\nu}(\phi_1,\phi_2)|\quad \phi_1,\phi_2\in C(2a,\nu)\}.
\end{align*}
Since $\left(\int \tilde \psi_i dm\right) $ is uniformly bounded with respect to $i$, we can upper bound correlations with $R_1' [\exp(R_2r^{j-i+1})-1]$, from which the thesis follows.
\end{proof}

\begin{proof}[Proof of Part (3) of Theorem \ref{thmB}]
We know that
\begin{align*}
\mb E[\psi_k]&=\int\psi_kdm\\
&=\int U_{\gamma_1}...U_{\gamma_k}(\psi)dm\\
&=\int \psi \mc L_{\gamma_k}...\mc L_{\gamma_1}1dm
\end{align*}
and, using Theorem \ref{thmB}, since for every $k>\bar n$ and sufficiently small perturbations $\|\mc L_{\gamma_k}...\mc L_{\gamma_1}1-\phi_0\|_0<\epsilon$
$$
\left|\mb E[\psi_k]-\int \psi \phi_0 dm\right|\leq \epsilon\mb E[|\psi|],
$$
and 
$$
\int \psi\phi_0dm-\epsilon\mb E[|\psi|]\leq\liminf_{n\rightarrow\infty}\frac{1}{n}\sum_k\mb E[\psi_k]\leq\limsup_{n\rightarrow\infty}\frac{1}{n}\sum_k \mb E[\psi_k]\leq \int \psi\phi_0dm+\epsilon\mb E[|\psi|].
$$
Thanks to Lemma \ref{DecCorLLN} we can directly apply Theorem \ref{WalkerLLN} to the random variables $X_k= \psi_k$. This implies that for almost every $x\in M$
$$
\lim_{n\rightarrow\infty}\frac{1}{n}\sum_{k=0}^{n-1}\psi_k(x)-\mb E[\psi_k]=0.
$$
which implies that 
$$
\int \psi\phi_0dm-\epsilon\mb E[|\psi|]\leq\liminf_{n\rightarrow\infty}\frac{1}{n}\sum_k\psi_k(x)\leq\limsup_{n\rightarrow\infty}\frac{1}{n}\sum_k \psi_k(x)\leq \int \psi\phi_0dm+\epsilon\mb E[|\psi|].
$$
\end{proof}

 \end{appendices}

\bibliographystyle{amsalpha}
\bibliography{bibliography}

\end{document}